%% file: scam.tex
\documentclass[a4paper,11pt,twoside,notitlepage]{article}  
\usepackage{amsmath,amsthm,amsfonts,euscript,stmaryrd,amscd,amssymb,makeidx}
\usepackage[nottoc]{tocbibind}
\usepackage{calrsfs}
\input{xy}\xyoption{all}
\usepackage{hyperref}    

\input{preamble}

\begin{document}
\author{Nicolas Stalder\footnote{Dept. of Mathematics, ETH Zurich, 8092 Zurich, Switzerland, nicolas@math.ethz.ch}}
\title{The Semisimplicity Conjecture for $A$-Motives}
\date{\today}
\maketitle
\begin{abstract}
We prove the semisimplicity conjecture for $A$-motives over finitely generated fields $K$.
This conjecture states that the rational Tate modules $\Vp(M)$ of a semisimple $A$-motive $M$
are semisimple as representations of the absolute Galois group of $K$.
This theorem is in analogy with known results for abelian varieties and Drinfeld modules, and
has been sketched previously by Akio Tamagawa. 

We deduce two consequences of the theorem for the algebraic monodromy groups $G_\frkp(M)$
associated to an $A$-motive $M$ by Tannakian duality. The first requires no semisimplicity
condition on $M$ and states that $G_\frkp(M)$ may be identified naturally with the Zariski closure
of the image of the absolute Galois group of $K$ in the automorphism group of $\Vp(M)$.
The second states that the connected component of $\Gp(M)$ is reductive if $M$ 
is semisimple and has a separable endomorphism algebra. 
\end{abstract}
\tableofcontents

\section{Introduction}

The aim of this article is to prove the following result, which is called the
\emph{semisimplicity conjecture for $A$-motives}.

\begin{thm}\label{thm:mainthm}
Let $K$ be a field which is finitely generated over a finite field.
Let $M$ be a semisimple $A$-motive over $K$ of characteristic $\iota$.
Let $\frkp\neq\ker\iota$ be a maximal ideal of $A$.
Then the rational Tate module $\Vp(M)$ associated to $M$ is semisimple as
$\frkp$-adic representation of the absolute Galois group $\Gal(\Ksep/K)$
of $K$.
\end{thm}

The strategy of our proof of the semisimplicity conjecture
is not original, it has been sketched by Tamagawa \cite{Tam95}.

Using the categorical machinery of my article \cite{Sta08},
the following consequences for the algebraic monodromy groups of
$A$-motives ensue formally from Theorem \ref{thm:mainthm}.

\begin{thm}\label{thm:secondmainthm}
Let $K$ be a field which is finitely generated over a finite field.
Let $M$ be an $A$-motive over $K$ of characteristic $\iota$, not necessarily
semisimple.
Let $\frkp\neq\ker\iota$ be a maximal ideal of $A$. Let $G_\frkp(M)$ be
the algebraic monodromy group of $M$, and let $\Gamma_\frkp(M)$ denote
the image of the absolute Galois group $\Gal(\Ksep/K)$ of $K$ in $\Aut_{\Fp}\big(\Vp(M)\big)$.
\begin{enumerate}
  \item The natural inclusion $\Gamma_\frkp(M)\subset\Gp(M)(\Fp)$
  has Zariski dense image.
  \item If $M$ is semisimple and its endomorphism algebra is separable,
  then the connected component of $G_\frkp(M)$ is a reductive group.
\end{enumerate}
\end{thm}

The concept of effective $A$-motives was invented by Anderson \cite{And86} in the
case $A=\Fq[t]$ for perfect $K$ under the name of $t$-motives. They may be viewed as analogues
of Grothendieck's pure motives, and even the conjectural heart of Voevodsky's
derived mixed motives, with the essential
difference that both the field of definition \emph{and} the ring of coefficients
of an $A$-motive are of positive characteristic. For an introduction to the theory of $A$-motives
we refer to the original source \cite{And86} and the books of Goss \cite{Gos96} and Thakur \cite{Tha04}.

The semisimplicity conjecture is an analogue of the Grothendieck-Serre conjecture which asserts the
semisimplicity of the etale cohomology groups of pure motives. This analogue has been proven only in
the case of abelian varieties, by Faltings \cite{Fal83} for fields of definition of characteristic
zero, and by Zarhin \cite{Zar76} for fields of definition of positive characteristic.

The semisimplicity conjecture is closely connected with two other conjectures,
the Tate conjecture and the isogeny conjecture. Only the conjunction of the Tate conjecture with the
semisimplicity conjecture allows us to deduce the consequences for the algebraic monodromy groups
of $A$-motives. The Tate conjecture characterises Galois-invariant endomorphisms of the associated
Tate modules. It has been proven independently by Tamagawa \cite{Tam94a} and Taguchi \cite{Tag95,Tag96} and will be
reproven in this article (Proposition \ref{prop:tateconj}).
The isogeny conjecture on the other hand is a fundamental finiteness statement which, as in the
case of abelian varieties, implies both the Tate conjecture and the semisimplicity conjecture.
For fields of definition of transcendence degree $\le 1$, the isogeny conjecture has been proven quite recently
by Pink \cite{Pin07}, using a different method. It seems that his results combined with
ours allow to deduce the isogeny conjecture for all finitely generated fields of
definition.

A special class of $A$-motives arises from Drinfeld modules. All such $A$-motives
are semisimple, and the semsimplicity conjecture for this class has been proven previously
by Taguchi in \cite{Tag91,Tag93} for fields of definition of transcendence degree $\le 1$.

We end the introduction with an overview of this article. In Section $2$ we
construct the rigid tensor category $\AMotK$ of $A$-motives in the spirit of Taelman \cite{Tae07},
containing the full subcategory $\AMotKeff$ of effective $A$-motives.
Inverting isogenies, we obtain the Tannakian category of $A$-isomotives. We introduce
the integral Tate module functors $\Tp$ with values in the categories of integral
$\frkp$-adic Galois representations $\RepApGalK$. They induce the rational
Tate module functors $\Vp$ with values in the Tannakian categories of rational $\frkp$-adic
Galois representations $\RepFpGalK$.
\[\xymatrix{
\AMotKeff \ar@{}[r]|{\subset} & \AMotK \ar[rr]^{\Tp}\ar[d] && \RepApGalK \ar[d]^{\Fp\otimes_{\Ap}(-)}\\
& \AIsomotK \ar[rr]_{\Vp} && \RepFpGalK
}\]

Section $3$ begins with the introduction of some terminology for semilinear
algebra: the notions of bold rings $\bR$, bold modules $\bM$, restricted bold modules
and bold scalar extension of modules from one bold ring to another.
Its main result concerns the study of bold scalar extension in a special situation.

In Section $4$ we show that the category of $A$-isomotives embeds into the
category $\FKModpet$ of $\frkp$-restricted bold modules over a certain bold ring $\bFK$.
We recall the classification of $\frkp$-adic Galois
representations in terms of the category $\FKpModet$ of $\frkp$-restricted $\bFKp$-modules, which employs
the functor $\Dp$ of Dieudonn\'e modules. In this translation to semilinear algebra, the functor induced
by the Tate module functor is of a rather simple form, it is the functor
$\bFKp\otimes_{\bFK}(-)$ of bold scalar extension from $\bFK$ to $\bFKp$. Following Tamagawa, we introduce an intermediate
bold ring $\bFK\subset\bFpK\subset\bFKp$, which allows to factor the above bold scalar extension
functor through the category of $\FpKModet$ of $\frkp$-restricted $\bFpK$-modules.
\[\xymatrix{
\AIsomotK \ar@{^(->}[d]_I \ar[rrrr]^{\Vp} &&&& \RepFpGalK \ar[d]_{\Dp}^{\isom}\\
\FKModpet \ar[rr]_{\bFpK\otimes_{\bFK}(-)} && \FpKModet \ar[rr]_{\bFKp\otimes_{\bFpK}(-)} && \FKpModet
}\]
The main result of Section $3$ then implies that the bold scalar extension functor $\bFpK\otimes_{\bFK}(-)$
maps semisimple objects to semisimple objects.

Sections $5$ and $6$ follow Tamagawa in constructing a certain bold ring $\bB$ which
induces a functor $\Cp$ from rational $\frkp$-adic Galois representations to $\frkp$-restricted $\bFpK$-modules.
All this is very much in the spirit of Fontaine theory, note however that we are dealing with
global Galois representations, not local Galois representations as in Fontaine theory.
\[\xymatrix{
&&\RepFpGalK \ar[d]^{\Dp} \ar[dll]_{\Cp}\\
\FpKModet \ar[rr]_{\bFKp\otimes_{\bFpK}(-)} && \FKpModet
}\]
The functor $\Cp$ has a variety of favourable properties. Among others, it allows to decide which Galois representations
arise from a $\frkp$-restricted $\bFpK$-module\footnote{Tamagawa calls such representations quasigeometric.} by a numerical criterion.
It also ensures that the bold scalar extension functor $\bFKp\otimes_{\bFpK}(-)$
maps semisimple objects to semisimple objects. Thereby, the proof of Theorem \ref{thm:mainthm}
is completed.

Finally, Section $7$ introduces the algebraic monodromy groups associated to $A$-isomotives
via Tannakian duality applied to the fibre functor $\Vp$ of Tate modules. We deduce
Theorem \ref{thm:secondmainthm} from Theorem \ref{thm:mainthm}, using results
from my article \cite{Sta08}.

This article as well as \cite{Sta08} are developments of my Ph.D. thesis. It is my
pleasure to thank Richard Pink for his guidance during my doctoral studies. I also wish
to thank Akio Tamagawa for helpful email exchanges, and encouraging me to publish
this article on my own.

\section{$A$-Isomotives}

Let $F$ be a global field of positive characteristic $p$, with finite field of constants $\Fq$
of cardinality $q$. Fix a finite non-empty set $\{\infty_1,\ldots,\infty_s\}$ of places of $F$,
the ``infinite'' places.
Denote by $A$ the subring of $F$ consisting of those elements integral outside the infinite places. Choose
a field $K$ containing $\Fq$, and set $\AK:=A\otimes_{\Fq}K$, this is a Dedekind ring. Choose
also an $\Fq$-algebra homomorphism $\iota:\:A\to K$, it corresponds to a prime ideal $\frkP_0$ of $\AK$
of degree $1$. If $\iota$ is injective, we say that the characteristic is generic. If not,
we say that the characteristic is special.

Let $\sigma_q$ denote the Frobenius endomorphism $c\mapsto c^q$ of $K$, and
let $\sigma$ denote the induced endomorphism $a\otimes c\mapsto a\otimes c^q$ of $\AK$.
For any $\AK$-module $M$, a \emph{$\sigma$-linear map} $\tau:\:M\to M$ is an additive map which
satisfies $\tau(r\cdot m)=\sigma(r)\cdot\tau(m)$ for all $(r,m)\in\AK\times M$.

Note that to give a $\sigma$-linear map $\tau:\:M\to M$ is equivalent to giving its
\emph{linearisation} $\taulin:\:\sigma_*M:=\AK\otimes_{\sigma,\AK} M\to M,\:r\otimes m\mapsto r\cdot\tau(m)$,
which is an $\AK$-linear map.

\begin{dfn}
An \emph{effective $A$-motive over $K$} (of characteristic $\iota$) is a finitely generated projective $\AK$-module $M$ together
with a $\sigma$-linear map $\tau:\:M\to M$ such that the support of $M/(\AK\cdot\tau M)$ is contained
in $\{\frkP_0\}$.
The \emph{rank} $\rk(M)$ of an effective $A$-motive $(M,\tau)$ is the rank of its underlying $\AK$-module $M$.
\end{dfn}

\begin{dfn}
Let $M$ and $N$ be effective $A$-motives over $K$. A \emph{homomorphism} $M\to N$ is an $\AK$-linear map
that commutes with $\tau$. An \emph{isogeny} is an injective homomorphism with torsion cokernel
(as a homomorphism of $\AK$-modules).
\end{dfn}

The category $\AMotKeff$ of effective $A$-motives over $K$ is an $A$-linear category.
While the kernels and cokernels of all homomorphisms exist categorically, it is not
an abelian category since the categorical kernel and cokernel of an isogeny are
both zero, even though not all isogenies are isomorphisms.

\begin{dfn}
Let $(M,\tau_M)$ and $(N,\tau_N)$ be effective $A$-motives over $K$. The \emph{tensor product} $M\otimes N$ of $M$ and $N$
is the effective $A$-motive consisting of the $\AK$-module $M\otimes_{\AK}N$ together with
the $\sigma$-linear map
\[\tau:\:M\otimes_{\AK}N\to M\otimes_{\AK}N,\quad m\otimes n\mapsto \tau_M(m)\otimes\tau_N(n).\]
\end{dfn}

Endowed with this tensor product, the category $\AMotKeff$ is an associative, commutative and
unital tensor category. The unit $\bbu$ is given by $\AK$ itself, equipped with the $\sigma$-linear
map $\sigma$ itself. However, it is not a rigid tensor category, since the dual of an effective $A$-motive
$M$ does not exist except if its $\taulin$ is bijective.

\begin{prop}\label{prop:forcompofamothoms}
Let $L,M,N$ be effective $A$-motives over $K$. If $L$ is of rank $1$, then the natural homomorphism
\[\Hom(M,N)\To\Hom(M\otimes L,N\otimes L),\quad f\mapsto f\otimes\id\]
is an isomorphism.
\end{prop}
\begin{rem}
If a dual $L^\vee$ of $L$ would exist in the category of effective $A$-motives,
then Proposition \ref{prop:forcompofamothoms} would be trivial: We could
simply ``twist back'' using $L^\vee$. This is true more generally for
invertible objects in tensor categories, and we will use this fact
in the following without further mention.
\end{rem}
\begin{proof}\hspace{-4pt}\footnote{Compare \cite[2.2.1, Lemma]{Tae07}.}\hspace{3pt}
The given homomorphism is induced by the bijective homomorphism
$\Hom_{\AK}(M,N)\to\Hom_{\AK}(M\otimes_{\AK}L,N\otimes_{\AK}L),\:f\mapsto f\otimes\id$ of the underlying $\AK$-modules,
so it is injective. An $\AK$-linear map $g=f\otimes 1:\:M\otimes_{\AK}N\to M\otimes_{\AK}N$
is a homomorphism of effective $A$-motives if $(f\circ\tau_M)\otimes\tau_L=(\tau_N\circ f)\otimes\tau_L$.
This implies that $f\circ\tau_M=\tau_N\circ f$, so $f$ is a homomorphism of effective $A$-motives,
as required.
\end{proof}

\begin{dfn}
An \emph{$A$-motive over $K$} is a pair $X=(M,L)$ consisting
of two effective $A$-motives over $K$ of which $L$ is of rank $1$.
\end{dfn}

\begin{dfn}
Let $(M',L')$ and $(M,L)$ be $A$-motives over $K$. A \emph{homomorphism}
$(M',L')\to (M,L)$ of $A$-motives is a homomorphism $M'\otimes L\to M\otimes L'$ of effective
$A$-motives over $K$. If the latter is an isogeny, then we say
that the given homomorphism of $A$-motives is an isogeny.
\end{dfn}

\begin{ex}
Let $X=(M,L)$ be an $A$-motive. For every $0\neq a\in A$, the homomorphism
$M\otimes_{\AK}L\to M\otimes_{\AK}L,\:m\otimes l\mapsto a\cdot m\otimes l$ is
an isogeny $[a]_X:\:X\to X$, the \emph{scalar isogeny} of $X$ induced by $a$.
\end{ex}

Given this definition of homomorphisms of $A$-motives, it is not completely
obvious how to compose two homomorphisms. We will use Proposition \ref{prop:forcompofamothoms}.
Let $X'=(M',L')$, $X=(M,L)$ and $X''=(M'',L'')$ be $A$-motives over $K$. We define the
composition of homomorphisms as follows, where the isomorphisms are given by Proposition
\ref{prop:forcompofamothoms} and $\to$ is the composition of homomorphisms of effective $A$-motives:
\[
\def\objectstyle{\scriptstyle}
\def\labelstyle{\scriptstyle}
\xymatrix{
\Hom(X',X)\times\Hom(X,X'') \ar@{=}[d]\\
\Hom(M'\otimes L,M\otimes L')\times\Hom(M\otimes L'',M''\otimes L) \ar[d]^{\isom}\\
\Hom(M'\otimes L\otimes L'',M\otimes L'\otimes L'')\times\Hom(M\otimes L'\otimes L'',M''\otimes L'\otimes L)\ar[d]\\
\Hom(M'\otimes L\otimes L'',M''\otimes L'\otimes L)\\
\Hom(M'\otimes L'',M''\otimes L') \ar@{=}[d]\ar[u]_{\isom}\\
\Hom(X',X'').
}\]

The category $\AMotK$ of $A$-motives over $K$ is an $A$-linear category.
Note that the direct sum of two $A$-motives $X'=(M',L')$ and $X=(M,L)$
is given by $X'\oplus X=\big((M'\otimes L)\oplus(M\otimes L'),L'\otimes L)$.

We have a natural functor from effective $A$-motives
to $A$-motives, mapping $M$ to $(M,\bbu)$.

\begin{dfn}
The \emph{tensor product} of two $A$-motives $X'=(M',L')$ and $X=(M,L)$
is the $A$-motive \[X'\otimes X=(M'\otimes M,L'\otimes L).\]
\end{dfn}

\begin{dfn}
Let $X=(M,L)$ be an $A$-motive, and let $d\ge 0$ be an integer.
The \emph{$d$-th exterior power} $\bigwedge^d X$ of $X$
is the $A$-motive $(\bigwedge^dM,L)$, where $\bigwedge^d M$
denotes the $d$-th exterior power of the $\AK$-module underlying
$M$ together with the unique $\sigma$-linear endomorphism such
that the homomorphism $\bigotimes_{\AK}^dM\to\bigwedge_{\AK}M$
is a homomorphism of $A$-motives.

We denote the second-highest and highest nontrivial exterior
powers of $X$ as $M^*:=\bigwedge^{\rk(M)-1}M$ and $\det(M):=\bigwedge^{\rk(M)}M$,
respectively.
\end{dfn}

\begin{prop}
The category $\AMotK$ of $A$-motives over $K$ is a rigid
$A$-linear tensor category, and the natural functor $\AMotKeff\to\AMotK$
is a fully faithful $A$-linear tensor functor.
\end{prop}
\begin{proof}
We suppress the details, remarking only that the
dual of an $A$-motive $X=(M,L)$ is given by $X^\vee:=(M^*\otimes L,\det M)$.
\end{proof}

Considering $\AMotKeff$ as a subcategory of $\AMotK$,
we note that an $A$-motive $X=(M,L)$ is the internal Hom $\mcH om(L,M)$
of the effective $A$-motives $M$ and $L$.

The category of $A$-motives is again not an abelian category. To obtain such a category,
we must invert those homomorphisms which have both zero kernel and zero cokernel, the isogenies.
We start by studying isogenies more carefully.

We will see that every isogeny is a factor of a scalar isogeny
(Proposition \ref{prop:isogfactstand}). This will allow us to
``invert isogenies'' by inverting scalar isogenies, technically
a simpler task.

\begin{dfn}
\begin{enumerate}
\item
A \emph{torsion $\bAK$-module} is a finitely-generated torsion $\AK$-module $T$
together with a $\sigma$-linear map $\tau:\:T\to T$. A homomorphism of torsion
$\bAK$-modules is a $\tau$-equivariant homomorphism of $\AK$-modules. The category
of torsion $\bAK$-modules is an $A$-linear abelian category, and has an
evident tensor product.
\item
We say that a torsion $\bAK$-module $(T,\tau)$ is \emph{of characteristic $\iota$}
if the supports of both kernel and cokernel
of $\taulin$ are contained in $\{\frkP_0\}$.
\end{enumerate}
\end{dfn}

Given an isogeny $f:\:M\to N$ of effective $A$-motives, the quotient $T:=N/f(M)$
in the category of $\AK$-modules inherits a $\sigma$-linear map, so
$T$ is a torsion $\bAK$-module. Note that it is of characteristic $\iota$.
If necessary, we denote $(T,\tau)$ by $\coker_{\bAK}(f)$.

\begin{dfn}
Let $f:\:M'\to M$ be an isogeny of effective $A$-motives, and set $(T,\tau):=\coker_{\bAK}(f)$.
The isogeny $f$ is \emph{separable} if $\taulin$ is bijective. The isogeny is \emph{purely inseparable} if $\tau$ is nilpotent.
We extend these two notions to isogenies of $A$-motives via the corresponding isogenies of effective $A$-motives.
\end{dfn}
%

With an eye towards our interest in isogenies of $A$-motives, we turn
to a discussion (Theorem \ref{thm:discusstorsionakmods}) of the structure
of the associated torsion $\bAK$-modules of characteristic $\iota$.

We intersperse a discussion of the connection of torsion $\bAK$-modules with bijective
$\taulin$ with Galois representations. The natural place for this would be later in the article, but it
will be useful in the proof of the next theorem.

\begin{dfn}
Let $\GalK:=\Gal(\Ksep/K)$ denote the absolute Galois group of $K$.
An \emph{$A$-torsion Galois representation} is an $A$-module $V$ of finite length
together with a group homomorphism $\rho:\:\GalK\to\Aut_A(V)$.
\end{dfn}

\begin{dfn}
\begin{enumerate}
  \item Let $(T,\tau)$ be a torsion $\bAK$-module such that $\taulin$ is bijective.
    We set $R_q(T,\tau):=(\Ksep\otimes_KT)^{\tau}$, taking $\tau$-invariants
    with respect to the diagonal action\footnote{``R'' for representation.}. Note that the action of $\GalK$ on $\Ksep$
    induces an action of $\GalK$ on $R_q(T,\tau)$.
  \item Let $(V,\rho)$ be an $A$-torsion Galois representation.
    We set $D_q(V,\rho):=(\Ksep\otimes_{\Fq}V)^{\GalK}$, taking $\GalK$-invariants
    with respect to the diagonal action\footnote{``D'' for Dieudonn\'e.}. Note that the $\sigma$-linear endomorphism $\sigma_q$
    of $\Ksep$ induces a $\sigma$-linear endomorphism $\tau$ of $D_q(V,\rho)$.
\end{enumerate}
\end{dfn}

\begin{prop}\label{prop:torsiongaloisrepsclassified}
Let $\GalK:=\Gal(\Ksep/K)$ denote the absolute Galois group of $K$.
The functors $D_q,R_q$ are quasi-inverse equivalences of $A$-linear rigid abelian
tensor categories:
\[\xymatrix{
\left(\!\!\left(\text{\begin{tabular}{c}$A$-torsion \\ Galois representations\end{tabular}}\right)\!\!\right) \ar@{}[r]|-{\isom} \ar@<2ex>[r]^-{D_q}&
\left(\!\!\left(\text{\begin{tabular}{c}torsion $\bAK$-modules \\ with bijective $\taulin$\end{tabular}}\right)\!\!\right) \ar@<2ex>[l]^-{R_q}
}\]
Moreover, the following is true:
\begin{enumerate}
  \item $\dim_KD(V,\rho)=\dim_{\Fq}V$ for every $A$-torsion Galois representation.
  \item The homomorphism $\Ksep\otimes_{\Fq} R_q(T,\tau)\to\Ksep\otimes_KT$ is an isomorphism
    for every torsion $\bAK$-module $(T,\tau)$ with bijective $\taulin$.
  \item The homomorphism $\Ksep\otimes_K D_q(V,\rho)\to\Ksep\otimes_{\Fq}V$ is an isomorphism
    for every $A$-torsion Galois representation $(V,\rho)$.
\end{enumerate}
\end{prop}
\begin{proof}
Forgetting the $A$-module structure of both sides, this is \cite[Proposition 4.1]{PiT06}
and its proof.
By naturality of that proposition, the statement of our proposition holds.
\end{proof}

\begin{thm}\label{thm:discusstorsionakmods}
Let $(T,\tau)$ be a torsion $\bAK$-module of characteristic $\iota$.
\begin{enumerate}
  \item If $\ker\iota=0$, then $\taulin$ is bijective.
  \item If $\ker\iota\neq 0$, then there exists a canonical filtration
    \[0\to(T',\tau')\to(T,\tau)\to(T'',\tau'')\to 0\]
    of $(T,\tau)$ by torsion $\bAK$-modules such that $\taulin'$ is bijective
    and $\tau''$ is nilpotent.
  \item If $\tau$ is nilpotent, then there exists a canonical filtration
    of $(T,\tau)$ by torsion $\bAK$-modules such that each successive subquotient
    is annihilated by $\tau$.
  \item We have $\Ann_A(T)\neq 0$.
\end{enumerate}
\end{thm}
\begin{proof}\hspace{-4pt}\footnote{I thank Gebhard B\"ockle for helping me simplify this proof.}\hspace{3pt}
(a): Since $\frkP_0$ lies over the generic prime of $A$, we have:
\begin{equation}\label{eqn:generic}
\text{The prime ideals $\sigma_*^m(\frkP_0)$ for $m\ge 0$ are pairwise different}.\end{equation}
Set $X:=\ker(\taulin)$ and $Y:=\coker(\taulin)$.
We consider the exact sequence of $\AK$-modules
\[0\To X\To\sigma_*T\arrover{\taulin}T\To Y\To 0.\]
To every finitely-generated torsion $\AK$-module
$N\isom\bigoplus_{\frka}\AK/\frka$ we may associate its characteristic ideal $\chi(N):=\prod\frka$.
We have $\dim_KX=\dim_KY$, so $\chi(X)=\chi(Y)=\frkP_0^n$ for some $n\ge 0$,
and
\begin{equation}\label{eqn:generic2}
\chi(\sigma_*T)=\chi(T).
\end{equation}

Now (\ref{eqn:generic2}) means that $\sigma_*$ permutes the (finitely many)
prime ideals lying in the support of $T$. Therefore, for every such prime ideal $\frkP$
in the support there exists an integer $m\ge 0$ such that $\sigma^m_*\frkP=\frkP$.
Now (\ref{eqn:generic}) excludes the possibility that $\frkP_0$ is contained in
the support of $T$. It follows that both $X$ and $Y$ are zero,
so $\taulin$ is indeed bijective.

(b): Note that $\im(\taulin^m)=\AK\cdot\tau^m(T)$. Since $T$ has finite length, this
chain of submodules becomes stationary and
$T':=\bigcap_{m\ge0}\im(\taulin^m)=\im(\taulin^n)$ for some $n\gg 0$.
In particular, the restriction of $\taulin$ to $T'$ is bijective, and
the induced $\sigma$-linear endomorphism of $T'':=T/T'$ is nilpotent.

(c): Clearly, $\taulin(T)\subset T$ is a $\tau$-invariant $\AK$-submodule.
The induced action of $\tau$ on the quotient $T/\taulin(T)$ is zero by construction.
Since $T$ has finite length, we may repeat this construction to obtain
a filtration with the desired properties.

(d): It is sufficient to prove the statement for the successive subquotients
of any chosen filtration of $(T,\tau)$ by torsion $\bAK$-modules. We use those
given by items (b) and (c).

If $\taulin$ is bijective, then the $A$-torsion Galois representation
associated by Proposition \ref{prop:torsiongaloisrepsclassified} has finite length
as $A$-module, so it has non-zero annihilator in $A$. Again by Proposition \ref{prop:torsiongaloisrepsclassified},
it follows that $T$ itself has non-zero annihilator in $A$.

If $\tau$ is zero and $T$ is non-zero, then $T=\coker\taulin$ has support
contained in $\{\frkP_0\}$. By (a) we have $\frkP_0\cap A=\ker\iota\neq 0$, so again
$T$ has non-zero annihilator in $A$.

Using (a,b,c) and the previous special cases, it follows that $\Ann_A(T)\neq 0$
for all torsion $\bAK$-modules $(T,\tau)$ of characteristic $\iota$.
\end{proof}

\begin{prop}\label{prop:isogfactstand}
Every isogeny is a factor of a scalar isogeny. More
precisely, let $f:\: X'\to X$ be an isogeny of $A$-motives over $K$. There exists an element $0\neq a\in A$,
and an isogeny $g:\:X\to X'$ such that $g\circ f=[a]_{X'}$ and $f\circ g=[a]_{X}$, so the following diagram commutes:
\[\xymatrix{
& X\ar[dr]^g\ar[rr]^{[a]_X} && X\\
X' \ar[ur]^f \ar[rr]_{[a]_{X'}} && X'\ar[ur]_f
}\]
In particular, the relation of isogeny is an equivalence relation.
\end{prop}
\begin{proof}
We may assume that both $X'$ and $X$ are effective $A$-motives.
Let $(T,\tau):=\coker_{\bAK}(f)$, a torsion $\bAK$-module of characteristic $\iota$.
By Theorem \ref{thm:discusstorsionakmods}(d), there exists an element $0\neq a\in A$ such that $a\cdot T=0$.
Therefore, $a\cdot X$ is contained in $f(X')\isom X$, so we obtain an isogeny $X\arrover{g}X'$
with $f\circ g=[a]_{X}$. Since $f$ is a homomorphism of $\bAK$-modules, we have
$f\circ g\circ f=[a]_{X}\circ f=f\circ [a]_{X'}$, so since $f$ is injective we obtain $g\circ f=[a]_{X'}$.
\end{proof}

We include the following consequence of Theorem \ref{thm:discusstorsionakmods},
it will not be needed in the following.

\begin{prop}\label{prop:factiso}
Let $X'\arrover{f}X''$ be an isogeny of $A$-motives.
\begin{enumerate}
  \item If $\ker\iota=0$, then $f$ is separable.
  \item If $\ker\iota\neq 0$, then there exist canonically an $A$-motive $X$ and a factorisation
  $f=f''\circ f'$, 
\[\xymatrix{
X'\ar[rr]^f \ar[dr]_{f'}&&X'',\\
&X\ar[ur]_{f''}
}\]
such that $f':\:X'\to X$ is a
separable isogeny and $f'':\:X \to X''$ is a purely
inseparable isogeny.
\end{enumerate}
\end{prop}
\begin{proof}
We may assume that all $A$-motives involved are effective.
Set $(T,\tau):=\coker_{\bAK}(f)$. If $\ker\iota=0$, then $\taulin$ is bijective
by Theorem \ref{thm:discusstorsionakmods}(a), so $f$ is separable.

(b): If $\ker\iota\neq 0$, Theorem \ref{thm:discusstorsionakmods}(b)
gives us a canonical filtration
\[0\to(T',\tau')\to(T,\tau)\to(T'',\tau'')\to 0\]
such that $\taulin'$ is bijective and $\tau''$ is nilpotent.
Letting $X$ be the inverse image of $T'$ in $X''$, we obtain an
effective $A$-motive such that $f$ factors as desired.
\end{proof}

\begin{dfn}
An \emph{$A$-isomotive over $K$} is an $A$-motive over $K$.
A homomorphism of $A$-isomotives is an $F$-linear combination
of homomorphisms of $A$-motives. More precisely, given two $A$-isomotives
$X',X$, we set
\[\Hom_{\AIsomotK}(X',X):=F\otimes_A\Hom_{\AMotK}(X',X),\]
where $\AIsomotK$ denotes the category of $A$-isomotives over $K$.

We might say that an $A$-isomotive is \emph{effective}
if it is isomorphic in $\AIsomotK$ to an effective $A$-motive.
\end{dfn}

\begin{thm}\label{thm:aboutaisomotk}
\begin{enumerate}
  \item The natural functor $\AMotK\to\AIsomotK$ is universal among
  $A$-linear functors with target an $F$-linear category and mapping isogenies
  to isomorphisms.
  \item $\AIsomotK$ is an $F$-linear rigid abelian tensor category.
\end{enumerate}
\end{thm}
\begin{proof}
(a): Our given functor is $A$-linear by definition. It maps
isogenies to isomorphisms by Proposition \ref{prop:isogfactstand}.
Let $\mcC$ be an $F$-linear category, and let $V:\:\AMotK\to\mcC$
be an $A$-linear functor which maps isogenies to isomorphisms.

It remains to show that there exists a unique $A$-linear functor $V':\:\AIsomotK\to\mcC$
extending $V$. Since $\AMotK$ and $\AIsomotK$ have the same objects, we
turn our attention to homomorphisms. Since scalar isogenies \emph{are} isogenies,
and $V$ \emph{does} map isogenies to isomorphisms, the desired extension
$V'$ exists and is unique.

(b): The category of $A$-isomotives is $F$-linear by construction.
It inherits a rigid tensor product from the category of $A$-motives. We must
show that it is abelian. For this, assume that $f:\:X'\to X$ is a
homomorphism of $A$-isomotives with vanishing categorical kernel and cokernel.
We may assume that both $X'$ and $X$ are effective $A$-isomotives.
By the definition of homomorphisms of $A$-isomotives, there exists an element
$0\neq a\in A$ such that $a\cdot f:\:X'\to X$ is a homomorphism of effective
$A$-motives. The categorical kernel and cokernel of $a\cdot f$ remain zero, since
multiplication by $a$ is an isomorphism. Clearly, this implies that $a\cdot f$
is injective, and $\coker_{\bAK}(a\cdot f)$ is a torsion $\bAK$-module. Therefore,
$a\cdot f$ is an isogeny, and Proposition \ref{prop:isogfactstand} gives an element
$0\neq b\in A$ and an isogeny $g:\:X\to X$ such that $(a\cdot f)\circ g$ and $g\circ(a\cdot f)$ are both
multiplication by $b$. Since multiplication by $b$ is an isomorphism in
$\AIsomotK$, this implies that $f$ is an isomorphism.
\end{proof}

\begin{dfn}
An $A$-motive $M$ is \emph{semisimple} if it is such
as an object of the category of $A$-isomotives.
\end{dfn}

We turn to $\frkp$-adic Galois representations. For the
remainder of this section, we introduce the following notation:
Let $\GalK:=\Gal(\Ksep/K)$ denote the absolute Galois group of $K$.
For every maximal ideal $\frkp$ of $A$, denote
the $\frkp$-adic completions of $A$ and $F$ by $\Ap$ and $\Fp$.

\begin{dfn}
\begin{enumerate}
  \item An \emph{integral $\frkp$-adic Galois representation} is
  a free $\Ap$-module of finite rank together with a continuous
  group homomorphism $\rho:\:\GalK\to\Aut_{\Ap}(V)$. Equipped with
  $\GalK$-equivariant $\Ap$-linear homomorphisms, we obtain the
  category $\RepApGalK$ of integral $\frkp$-adic Galois representations.
  \item A \emph{rational $\frkp$-adic Galois representation} is
  a finite-dimensional $\Fp$-vector space together with a continuous
  group homomorphism $\rho:\:\GalK\to\Aut_{\Fp}(V)$. Equipped with
  $\GalK$-equivariant $\Ap$-linear homomorphisms, we obtain the
  category $\RepFpGalK$ of rational $\frkp$-adic Galois representations.
\end{enumerate}
\end{dfn}

\begin{dfn}
Let $\frkp\neq\ker\iota$ be
a maximal ideal of $A$, and let $\AKsepp:=\varprojlim_n \left((A/\frkp^n)\otimes_{\Fq}\Ksep\right)$
denote the completion of $A\otimes_{\Fq}\Ksep$ at $\frkp$.
For every $A$-motive $X=(M,L)$ over $K$:
\begin{enumerate}
  \item  The \emph{integral Tate module}
  of $X$ at $\frkp$ is the $\Ap$-module
    \[\Tp(X):=\left(\AKsepp\otimes_{\AK}M\right)^\tau\otimes_{\Ap}\left((\AKsepp\otimes_{\AK}L)^\tau\right)^\vee,\]
  with $\tau$-invariants taken with respect to the natural diagonal $\sigma$-linear endomorphism,
  equipped with the induced action of $\GalK$.
  \item The \emph{rational Tate module} of $X$ at $\frkp$ is the $\Fp$-vector space
    \[\Vp(X):=\Fp\otimes_{\Ap}\Tp(X),\]
    equipped with the induced action of $\GalK$.
\end{enumerate}
\end{dfn}

\begin{dfn}\label{dfn:SRff}
Let $R\to S$ be a homomorphism of unital rings, $\mcC$ an $R$-linear category,
and $\mcD$ an $S$-linear category. An $R$-linear functor $V:\:\mcC\to\mcD$ is
\emph{$S/R$-faithful} (resp., \emph{$S/R$-fully faithful}) if the natural homomorphism
  \[S\otimes_R\Hom_{\mcC}(X,Y)\to\Hom_{\mcD}(VX,VY)\]
  is injective (resp., bijective) for all objects $X,Y$ of $\mcC$.
\end{dfn}

\begin{prop}
 Let $\frkp\neq\ker\iota$ be
a maximal ideal of $A$.
\begin{enumerate}
  \item $\Tp$ is an $A$-linear tensor functor with values in integral $\frkp$-adic representations,
  which is $\Ap/A$-faithful and preserves ranks.
  \item $\Vp$ extends uniquely to an $F$-linear functor with values in rational $\frkp$-adic representations,
  again denoted as $\Vp$, such that the following diagram commutes:
    \[\xymatrix{
    \AMotK \ar[d] \ar[rr]^-{\Tp} && \RepApGalK \ar[d]^{\Fp\otimes_{\Ap}(-)} \\
    \AIsomotK     \ar[rr]^-{\Vp} && \RepFpGalK
    }\]
  \item $\Vp$ is an exact $F$-linear tensor functor which is $\Fp/F$-faithful and preserves ranks.
\end{enumerate}
\end{prop}
\begin{proof}
(a): Let us first consider the restriction of $\Tp$ to effective $A$-motives, it
maps a given effective $A$-motive $M$ to
\[\Tp(M)=(\AKsepp\otimes_{\AK}M)^\tau=\varprojlim_n\left((M\otimes_K\Ksep)/\frkp^n\right)^\tau. \]
Note that the assumption that $\frkp\neq\ker\iota$ implies
that the linearisation of the $\sigma$-linear endomorphism of $(M\otimes_K\Ksep)/\frkp^n$ is bijective.
By applying Proposition \ref{prop:torsiongaloisrepsclassified} to $\Ksep$ and $(M\otimes_K\Ksep)/\frkp^n$,
we see that $\left((M\otimes_K\Ksep)/\frkp^n\right)^\tau$ is a free $A/\frkp^n$-module of rank
$\rk(M)$. It follows that $\Tp(M)$ is an integral $\frkp$-adic Galois representation
of rank $\rk(M)$. Using Proposition \ref{prop:torsiongaloisrepsclassified} again,
it follows that the restriction of $\Tp$ to $\AMotKeff$ is an $A$-linear tensor functor with values
in integral $\frkp$-adic representations which preserves ranks. By construction, this implies
that $\Tp$ itself has these properties.

It remains to show that $\Tp$ is $\Ap/A$-faithful. Let $M,N$ be $A$-motives.
We may assume that both are effective. Note that we have a natural
inclusion $\Ap\otimes_{\Fq}\Ksep\subset\AKsepp$. It follows that
we have a natural inclusion
\[(\Ap\otimes\Ksep)\otimes_{\AK}\Hom_{\AK}(M,N)\subset\AKsepp\otimes_{\AK}\Hom_{\AK}(M,N).\]
On both sides, the left exact functors $(-)^{\GalK}$ of Galois-invariants and
$(-)^\tau:=\ker(\tau_N\circ(-)-(-)\circ\tau_M)$ of $\tau$-invariants act, and the two actions commute.
Therefore,
\[\left((\Ap\otimes\Ksep)\otimes_{\AK}\Hom_{\AK}(M,N)\right)^{\GalK,\tau}\subset\left(\AKsepp\otimes_{\AK}\Hom_{\AK}(M,N)\right)^{\tau,\GalK},\]
so
\[\Ap\otimes_A\Hom(M,N)^\tau\subset\Hom_A(\Tp M,\Tp N)^\GalK,\]
which means that $\Ap\otimes_A\Hom(M,N)\to\Hom_{\GalK}(\Tp M,\Tp N)$ is
injective, as desired.

(b): Since scalar isogenies are mapped to isomorphisms in $\RepFpGalK$,
$\Vp$ extends to an $F$-linear functor on $\AIsomotK$ with values in rational
$\frkp$-adic Galois representations.

(c): Now item (a) implies that $\Vp$ is an $\Fp/F$-fully faithful tensor functor,
and preserves ranks. This last property implies that $\Vp$ is exact.
\end{proof}

\begin{cor}
\begin{enumerate}
  \item For every two $A$-motives $M,N$, the $A$-module of homomorphisms $\Hom_{\AMotK}(M,N)$ is
  finitely-generated and projective.
  \item For every two $A$-isomotives $X,Y$, the $F$-vector space of homomorphisms $\Hom_{\AIsomotK}(X,Y)$
  is finite-dimensional.
  \item Every $A$-isomotive has a composition series of finite length.
\end{enumerate}
\end{cor}
\begin{proof}
(b): Since $\Hom_{\GalK}(\Vp X,\Vp Y)$ is finite $\Fp$-dimensional, so
is $\Fp\otimes_F\Hom_{\AIsomotK}(X,Y)$ by $\Fp/F$-faithfulness of $\Vp$. This
implies the desired statement.

(a): If we show that $\Hom_{\AMotK}(M,N)$ is torsion-free, item (a) follows
from item (b). However, $\Hom_{\AMotK}(M,N)=(M^\vee\otimes N)^\tau$ is a submodule
of the torsion-free $A$-module $M^\vee\otimes N$, so we are done.

(c): Since $\Vp$ is faithful, it maps non-zero objects to non-zero objects.
Therefore, the length of an $A$-isomotive is bounded by the length of its
Tate module. Since the latter is of finite length, so is the former.
\end{proof}

\section{Some Semilinear Algebra}

We begin this section by introducing the notion ``semisimple
on objects'' for functors, a categorical generalisation of the
statement of Theorem \ref{thm:mainthm}, and discuss how this
property combines with the notion of ``relative full faithfulness'',
introduced in Definition \ref{dfn:SRff}.

We then introduce some terminology for semilinear algebra, and prove
a theorem on bold scalar extension of restricted modules for a certain
class of bold rings. The reader may choose to skip to Section $4$ after reading
the statement of Theorem \ref{thm:boldscalarext}, to see how it is employed.

\begin{dfn}
Let $\mcA,\mcB$ be abelian categories. An exact functor $V:\:\mcA\to\mcB$ is \emph{semisimple on objects}
if it maps semisimple objects of $\mcA$ to semisimple objects of $\mcB$.
\end{dfn}

We intersperse a proposition which exemplifies nicely how the properties of
being ``relatively'' full faithful and being semisimple on objects combine.

\begin{prop}\label{prop:relfullfaithfulisnonsemisimpleonobjects}
Let $F'/F$ be a field extension, $\mcA$ an $F$-linear
abelian category, and $\mcB$ an $F'$-linear abelian category.
Consider an $F'/F$-fully faithful $F$-linear exact functor $V:\:\mcA\to\mcB$.
For every object $X$ of $\mcA$, if $V(X)$ is semisimple in $\mcB$, then
$X$ is semisimple in $\mcA$.\footnote{In other words, $V$ maps non-semisimple objects of $\mcA$
to non-semisimple objects of $\mcB$, we will use this
in the proof of Proposition \ref{prop:boldscalarext2}.}
\end{prop}
\begin{proof}
Assume that
\[\alpha:\quad 0\to X'\to X\to X''\to 0\]
is a short exact sequence
in $\mcA$ such that the exact sequence $V(\alpha)$ splits in $\mcB$.
We must show that $\alpha$ splits, and for this it suffices to show that $\id_{X''}$ is in the image
of the natural homomorphism $\Hom_{\mcA}(X'',X)\to\Hom_{\mcA}(X'',X'')$. This image
coincides with the intersection of $\Hom_{\mcA}(X'',X'')$ and the image of
the natural homomorphism $F'\otimes_F\Hom_{\mcA}(X'',X)\to F'\otimes_F\Hom_{\mcA}(X'',X'')$.
By $F'/F$-full faithfulness, we may identify this latter image with
the image of the natural homomorphism $\Hom_{\mcB}(V(X''),V(X))\to\Hom_{\mcB}(V(X''),V(X''))$.
By assumption, $\id_{V(X'')}=V(\id_{X''})$ is an element of this image, and under our natural
identifications it is also clearly an element of $\Hom_{\mcA}(X'',X'')$, therefore we are done.
\end{proof}

We turn to some general terminology for semilinear algebra.

\begin{dfn}
A \emph{bold ring} $\bR$ is a unital commutative ring $R$ equipped with a unital ring endomorphism $\sigma:\:R\to R$.
The \emph{coefficient ring} of $\bR$ is its subring $R^\sigma:=\{r\in R:\:\sigma(r)=r\}$
of $\sigma$-invariant elements. 

A \emph{homomorphism} $\bS\to\bR$ of bold rings
is a ring homomorphism that commutes with $\sigma$.
It induces a homomorphism $S^\sigma\to R^\sigma$ of coefficient rings.
\end{dfn}

\begin{dfn}
Let $\bR$ be a bold ring. A (bold) \emph{$\bR$-module} $\bM$ is an $R$-module $M$ together
with a $\sigma$-linear endomorphism $\tau:\:M\to M$.

A \emph{homomorphism} $\bM\to\bN$ of $\bR$-modules is an $R$-module homomorphism that commutes with $\tau$.
The \emph{tensor product} $\bM\otimes_{\bR}\bN$
of $\bM=(M,\tau_M)$ and $\bN=(N,\tau_N)$ is the $R$-module $M\otimes_RN$ together with the $\sigma$-linear
endomorphism
\[M\otimes_RN\to M\otimes_RN,\quad m\otimes n\mapsto \tau_M(m)\otimes\tau_N(n).\]
The category $\bRMod$ of
$\bR$-modules is an $R^\sigma$-linear abelian tensor category.
\end{dfn}

\begin{dfn}
Let $\bS\arrover{f}\bR$ be a homomorphism of bold rings. \emph{Bold scalar extension} from
$\bS$ to $\bR$ is the functor $\bSMod\to\bRMod$ mapping an $\bS$-module $\bM$ to $\bR\otimes_{\bS}\bN$
and a homomorphism $h$ of $\bS$-modules to $\id_R\otimes h$.
\end{dfn}

Recall from Section $2$ that the $\sigma$-linear endomorphism $\tau$ of a module
$\bM$ over a bold ring $\bR=(R,\sigma)$ corresponds to a unique $R$-linear
homomorphism $\taulin:\:\sigma_*M:=R\otimes_{\sigma,R}M\to M$, its linearisation.

\begin{dfn}
Let $\bR$ be a bold ring.
\begin{enumerate}
  \item An $\bR$-module $\bM=(M,\tau)$ is \emph{restricted}
				if $M$ is a finitely generated projective $R$-module and $\taulin$ is bijective.
  \item Let $\bS\arrover{f}\bR$ be a homomorphism of bold rings. An $\bR$-module $\bM$ is
\emph{$f$-restricted} if there exist a restricted $\bS$-module $\bN$ and an isomorphism
$\bM\isom\bR\otimes_{\bS}\bN$ of $\bR$-modules. Clearly, this implies that
$\bM$ is restricted in the sense of (a).
\end{enumerate}
\end{dfn}

Let $\Fq,K,\sigma_q$ be as in Section $2$, so $\Fq$ is a finite field and $K$ is a field
containing $\Fq$. In this section (but not the next) $F/\Fq$ may be any field
extension, that is,
\begin{tabular}{|c|}
\hline
we \emph{drop} the assumption that $F$ is a global field.\\
\hline
\end{tabular}
Besides yielding more generality, this allows us more flexibility in the proofs.

Let $\FK=\Frac(F\otimes_{\Fq}K)$ denote the total ring of fractions
of $F\otimes_{\Fq}K$. The bold ring $\bFK$ is given by $\FK$ together
with the endomorphism $\sigma=\sigma_{\FK}=\Frac(\id\otimes\sigma_q)$ induced
by $\sigma_q$. If $F'/F$ is a field extension, the bold ring $\bFFK$
with underlying ring $\FFK=\Frac(F'\otimes_{\Fq}K)$ is defined analogously,
and we have a bold scalar extension functor $\bFFK\otimes_{\bFK}(-)$ from
$\bFK$-modules to $\bFFK$-modules.

\begin{lem}\label{lem:onbffkmodules}
Assume that the number of roots of unity of $K$ is finite.
\begin{enumerate}
  \item The ring $\FK$ is a finite product of pairwise isomorphic fields.
  \item The underlying $\FK$-module of every restricted $\bFK$-module is free.
\end{enumerate}
\end{lem}
\begin{proof}
Let $\bbF_F$ and $\bbF_K$ denote the algebraic closures of $\Fq$ in $F$ and $K$,
respectively. If $\bbF_F=\bbF_{q^r}$ and $\bbF_K=\bbF_{q^s}$ are both finite,
then
\[\bbF_F\otimes_{\Fq}\bbF_K\isom\left(\bbF_{q^{\lcm(r,s)}}\right)^{\times\gcd(r,s)}\]
and $\sigma=\id\otimes\sigma_q$ corresponds to an endomorphism
of the product which permutes the factors transitively. This implies
that every restricted $(\bbF_F\otimes_{\Fq}\bbF_K,\id\otimes\sigma_q)$-module
has an underlying $\bbF_F\otimes_{\Fq}\bbF_K$-module which is projective of
\emph{constant} rank, and hence free. Hereby, items (a) and (b) are proven for
$F$ and $K$ both finite.

If $\bbF_F$ is infinite, then it is an algebraic closure of $\Fq$ and
\[\bbF_F\otimes_{\Fq}\bbF_K\isom(\bbF_F)^{\times\dim_{\Fq}\bbF_K}\]is a product of
pairwise isomorphic fields. It follows from the above that the
endomorphism corresponding to $\sigma=\id\otimes\sigma_q$ permutes the factors
transitively, so again we have items (a) and (b) for $F$ and $K$ both algebraic.

In the general case, $\bbF_F\otimes_{\Fq}\bbF_K\isom\bbF^r$ for an
algebraic extension $\bbF/\bbF_q$ and an integer $r\ge 1$. Then \cite[Theorem 8.50]{Jac90}
shows that $F\otimes_{\bbF_F}\bbF\otimes_{\bbF_K}K$ is a domain,
which implies that
\[\FK\isom\Frac(F\otimes_{\bbF_F}\bbF\otimes_{\bbF_K}K)^{\times r}\]
is a product of pairwise isomorphic fields. Tracing through these identifications,
we see that $\sigma_{\FK}$ permutes these fields transitively, so we obtain
items (a) and (b) in general.
\end{proof}

\begin{prop}\label{prop:subquotclosedbFK}
Assume that the number of roots of unity of $K$ is finite.
The full subcategory of restricted $\bFK$-modules is closed under
subquotients and tensor products in the category of all $\bFK$-modules. In particular,
it is an $F$-linear rigid abelian tensor category.
\end{prop}
\begin{proof}
Let $\bM=(M,\tau)$ be a restricted $\bFK$-module,
and consider an exact sequence
\[0\to(M',\tau')\to\bM\to(M'',\tau')\to 0\]
of $\bFK$-modules. Both $M'$ and $M''$ are finitely generated $\FK$-modules
since $\FK$ is Noetherian, and both are projective $\FK$-modules
since $\FK$ is a product of fields by Lemma \ref{lem:onbffkmodules}(a).
Since $\taulin:\:\sigma_*M\to M$ is bijective, the Snake Lemma
implies that $\taulin'$ is injective and $\taulin''$ is surjective.
By Lemma \ref{lem:onbffkmodules}, this implies that both $\taulin'$ and
$\taulin''$ are bijective. Therefore, both $(M',\tau')$ and $(M'',\tau'')$
are restricted $\bFK$-modules as claimed.

We suppress the easy proof that the tensor product of restricted
$\bFK$-modules is restricted. It follows that the full subcategory
of restricted $\bFK$-modules is an $F$-linear abelian tensor category,
since $\bFK$-$\Mod$ is. One checks that the dual of a restricted
$\bFK$-module $(M,\tau)$ is given by $M^\vee:=\Hom_{\AK}(M,\AK)$ together
with the $\sigma$-linear endomorphism mapping
$f\in M^\vee$ to
$\tau_{M^\vee}(f):=\sigma_{\mathrm{lin}}\circ\sigma_*(f)\circ(\taulin)^{-1}$.
It follows that the category of restricted $\bFK$-modules is a rigid
tensor category.
\end{proof}

We turn to the main theorem of this section, its proof will occupy the
remainder of the section. To state it, we recall
the algebraic concept of separability.

\begin{dfn}\label{dfn:fieldexseper}
A field extension $F'/F$ is \emph{separable}
if for every field extension $F''\supset F$ the ring
$F'\otimes_FF''$ is reduced (contains no nilpotent elements).
\end{dfn}

\begin{rem}
An algebraic field extension $F'/F$ is separable in the
sense of Definition \ref{dfn:fieldexseper} if and only if
it is separable in the usual sense. If $F''/F'/F$ is a
tower of field extensions such that $F''/F$ is separable,
then $F'/F$ is separable as well.
\end{rem}

\begin{thm}\label{thm:boldscalarext}
Let $F'/F/\Fq$ be a tower of field extensions.
Assume that the number of roots of unity of $K$ is finite.
The restriction of the functor of bold scalar extension $\bFFK\otimes_{\bFK}(-)$
to restricted $\bFK$-modules is:
\begin{enumerate}
  \item $F'/F$-fully faithful and,
  \item if $F'/F$ is a separable field extension, it is semisimple on objects.
\end{enumerate}
\end{thm}

We turn first to the proof of item (a) of Theorem \ref{thm:boldscalarext}.

\begin{prop}\label{prop:boldscalarext1}
Let $F'/F/\Fq$ be a tower of field extensions.
Assume that the number of roots of unity of $K$ is finite.
The restriction of the functor of bold scalar extension $\bFFK\otimes_{\bFK}(-)$
to restricted $\bFK$-modules is $F'/F$-fully faithful.
\end{prop}
\begin{proof}
Let $\bM,\bN$ be restricted $\bFK$-modules,
and set $\bX:=\bM^\vee\otimes_{\bFK}\bN$. Since
$\Hom_{\bFK}(\bM,\bN)=\bX^\tau$ and
$\Hom_{\bFFK}(\bFFK\otimes_{\bFK}\bM,\bFFK\otimes_{\bFK}\bN)=(\bFFK\otimes_{\bFK}\bX)^\tau$,
it is sufficient to prove that
\begin{equation}\label{eqn:reductiontotauinvariants}
F'\otimes_F\bX^\tau\to(\bFFK\otimes_{\bFK}\bX)^\tau
\end{equation}
is bijective for all restricted $\bFK$-modules $\bX$. We
set $\bX':=\bFFK\otimes_{\bFK}\bX$.

Since the homomorphism
$F'\otimes_F\FK\to F'_K=\Frac(F'\otimes_F\FK)$ is injective and the functor
$(-)^\tau$ is left-exact, the homomorphism of (\ref{eqn:reductiontotauinvariants})
is injective. We must show that it is surjective!

Moreover, we may assume that $F'\supset F$ is finitely
generated, since for every element $x'\in(\bX')^\tau$ there
exists a finitely generated field extension $F'\supset F^0\supset F$
such that $x'$ lies in $(\boldsymbol{X^0})^\tau$,
where $\boldsymbol{X^0}:=\boldsymbol{F_K^0}\otimes_{\bFK}\bX$
with $\boldsymbol{F_K^0}:=\Frac(F^0\otimes_F\FK,\id\otimes\sigma_q)$.

All in all, the theorem reduces to proving the surjectivity
of (\ref{eqn:reductiontotauinvariants}) for the two special cases of
$F'\supset F$ finite, and $F'\supset F$ purely transcendental of transcendence degree $1$.
The first is easy, since if $F'/F$ is finite, then $F'\otimes_F\bFK\isom\bFFK$,
and hence
\[F'\otimes_F\bM^\tau=(F'\otimes_F\bFK\otimes_{\bFK}\bM)^\tau\isom(\bFFK\otimes_{\bFK}\bM)^\tau\]
as claimed. The second is dealt with in the following Proposition \ref{prop:resinvtrans1}.
\end{proof}

\begin{prop}\label{prop:resinvtrans1}
If $F'=F(X)$ is a purely transcendental extension of $F$ of transcendence degree $1$
and $\bX$ is a restricted $\bFK$-module, then $F'\otimes_F\bX^\tau\to(\bFFK\otimes_{\bFK}\bX)^\tau$
is surjective.
\end{prop}

For the proof of Proposition \ref{prop:resinvtrans1}, we use a slightly
extended notion of ``denominators''. By Lemma \ref{lem:onbffkmodules}(a),
the ring $\FK=Q^{\times s}$ for some field $Q$. We set
$\FK[X]:=Q[X]^{\times s}$ and $\FK(X):=\Frac(F(X)\otimes_F\FK)=Q(X)^{\times s}$.

For $f\in\FK(X)$, we define the \emph{denominator} $\den(f)$ of $f$
componentwise, as the $s$-tuple of the usual (monic) denominators of its $s$ components.
Similarly, for $f,g\in\FK(X)$, we define the \emph{least common multiple}
$\lcm(f,g)$ of $f$ and $g$ componentwise, as the $s$-tuple of the usual
(monic) least common multiples of their corresponding components.

Clearly, for $f,g\in\FK(X)$ the following relation holds,
where $\mid$ denotes componentwise divisibility in $\FK[X]$:
\begin{equation}\label{eqn:denlcm}
\den(f+g) \mid \lcm(\den f,\den g).
\end{equation}

We may now characterise the subring $F(X)\otimes_F\FK$ of $\FK(X)$.

\begin{lem}\label{lem:resinvtrans1lem2}
We have
\[F(X)\otimes_F\FK=\left\{f\in\FK(X):\:\begin{array}{c}\den(f)\mid g\\\text{for some $g\in F[X]\smallsetminus\{0\}$}\end{array}\:\right\},\]
\end{lem}
\begin{proof}
$\subset$: Assume that $f$ is an element of $F(X)\otimes_F\FK$.
We may write $f=\sum_{i=1}^m(a_i/b_i)\otimes\lambda_i$ for elements
$\lambda_i\in\FK$ and $a_i,b_i\in F[X]$ with $b_i\neq 0$. By (\ref{eqn:denlcm}), 
$\den(f)$ divides $d:=\prod_{i=1}^m b_i$, an element of $F[X]\smallsetminus\{0\}$ as claimed.

$\supset$: Assume that $f$ is an element of $\FK(X)$ which divides
a non-zero element $g\in F[X]$. This means that there exists an element $h\in\FK[X]$
such that $g=\den(f)\cdot h$. We have $f=f'/\den(f)$ for $f':=f\den(f)\in\FK[X]$.
Therefore $f=(f'h)/(\den(f)h)$ with $1/(\den(f)h)=1/g\in F(X)$ and
$f' h\in\FK[X]\subset F(X)\otimes_F\FK$, which implies our claim that $f$ is an element of $F(X)\otimes_F\FK$.
\end{proof}

Given a vector $v=(v_j)\in\FK(X)^r$ for some $r\ge 1$, we set $\den(v)=\lcm_j(\den v_j)$.

\begin{lem}\label{lem:resinvtrans1lem1}
Given two integers $m,n\ge 1$, a matrix $A\in\Mat_{m\times n}(\FK)$
and a vector $v\in\FK(X)^{\oplus n}$, we have
\[\den\left(Av\right)\mid\den(v).\]
In particular, if $m=n$ and $A$ is invertible, then $\den(Av)=\den(v)$.
\end{lem}
\begin{proof}
We suppress the easy proof of the divisibility statement, which
is clear intuitively.

In case $m=n$ and $A$ is invertible, we may additionally apply
this divisibility statement to the matrix $A^{-1}$ and
the vector $Av$. We obtain $\den(v)=\den\big(A^{-1}(Av)\big)\mid\den(Av)$.
Since both $\den(Av)$ and $\den(v)$ have monic components, we
infer that $\den(Av)=\den(v)$.
\end{proof}

\begin{proof}[Proof of Proposition \ref{prop:resinvtrans1}]
By Lemma \ref{lem:onbffkmodules}(a),
$X=F_K^r$ for an integer $r\ge 0$ and $\tau=\Delta\circ\sigma$
for a certain matrix $\Delta\in\GL_r(\FK)$.

Assume that $x'\in\FK(X)\otimes_{\FK}X$ is $\tau$-invariant,
so $x'=(x_i')_i\in\FK(X)^r$ and $x'=\Delta\big(\sigma(x')\big)$.
By Lemma \ref{lem:resinvtrans1lem1} applied to the invertible matrix $\Delta$
and the vector $\sigma(x')$, we obtain that $\den(x')=\den\big(\sigma(x')\big)$,
and this latter vector clearly coincides with $\sigma\big(\den(x')\big)$.
Therefore, $\den(x')=\sigma\big(\den(x'))$ is an element of $F[X]$.
Since $\den(x_i')\mid\den(x')$ by definition, all $x_i'$ are elements of
$F'\otimes_F\FK$ by Lemma \ref{lem:resinvtrans1lem2}, and so
$x'\in F'\otimes_F\bX^\tau$, as claimed.
\end{proof}

We now turn to the proof of item (b) of Theorem \ref{thm:boldscalarext}.

\begin{prop}\label{prop:boldscalarext2}
Let $F'/F/\Fq$ be a tower of field extensions.
Assume that $F'/F$ is separable and the number of roots of unity of $K$ is finite.
The restriction of the functor of bold scalar extension $\bFFK\otimes_{\bFK}(-)$
to restricted $\bFK$-modules is semisimple on objects.
\end{prop}
\begin{proof}
As in the proof of Proposition \ref{prop:boldscalarext1},
we start by reducing to the case where $F'/F$ is finitely
generated: If $\bM$ is a semisimple restricted $\bFK$-module
but $\bM':=\bFFK\otimes_{\bFK}\bM$ is not semisimple, then
there exists a non-split short exact sequence
\begin{equation}\label{eqn:thenonsplitseqforpropboldscalarext2}
0\to\bM_1'\arrover{f}\bM'\arrover{g}\bM_2'\to 0.
\end{equation}
Clearly, there exists a finitely generated field extension,
$F'\supset F^0\supset F$ such that $\bM_1',\bM_2',f,g$ are
defined over $\boldsymbol{F_K^0}=\Frac(F^0\otimes_F\FK,\id\otimes\sigma_q)$.
The short exact sequence inducing (\ref{eqn:thenonsplitseqforpropboldscalarext2})
must be non-split by Propositions \ref{prop:relfullfaithfulisnonsemisimpleonobjects}
and \ref{prop:boldscalarext1}. Thereby, we would find a contradiction
to Proposition \ref{prop:boldscalarext2} for finitely generated
field extensions. Note that $F^0/F$ is separable since $F'/F$ is.

The same argument shows that the proof of our proposition
reduces to the special cases of finite separable field extensions
and purely transcendental field extensions of transcendence degree $1$.
We deal with these cases separately in the following two propositions.
Note that it is sufficient to show that the bold scalar extension
of a simple restricted $\bFK$-module is semisimple, since bold
scalar extension is an additive functor.
\end{proof}

\begin{prop}
Assume that the number of roots of unity of $K$ is finite.
Let $F'/F/\Fq$ be a tower of field extensions such that
$F'/F$ is finite separable, and let $\bM$ be a simple
restricted $\bFK$-module. Then $\bM':=\bFFK\otimes_{\bFK}\bM$
is semisimple.
\end{prop}
\begin{proof}
We start with the case where $F'/F$ is a finite \emph{Galois}
extension, and set $\Gamma:=\Gal(F'/F)$. Assume that
$\bS'\subset\bM'$ is a simple $\bFFK$-submodule. Set
\[\bX':=\sum_{g\in\Gamma}g\bS'\subset\bM'.\]
This $\bFFK$-module is $\Gamma$-invariant, so
$\bX'=\bFFK\otimes_{\bFK}\bX$ for some $\bFK$-submodule
of $\bX\subset\bM$. Since $\bS$ is simple and $\bS'\neq 0$,
we see that $\bX=\bM$ and so $\bM'=\sum_{g\in\Gamma}g\bS'$
is semisimple as a sum of simple objects.

In the general case, let $F''/F$ denote a Galois
closure of $F'/F$, and consider a simple restricted
$\bFK$-module $\bM$. By what we have proven,
$\bM'':=\bFFFK\otimes_{\bFK}\bM$ is semisimple.
Now Proposition \ref{prop:relfullfaithfulisnonsemisimpleonobjects} shows
that $\bM':=\bFFK\otimes_{\bFK}\bM$ is semisimple, since
we have already proven Proposition \ref{prop:boldscalarext1}.
\end{proof}

\begin{prop}
Assume that the number of roots of unity of $K$ is finite.
Let $F/\Fq$ be a field extension, consider $F'=F(X)$
and let $\bM$ be a simple restricted $\bFK$-module.
Then $\bM':=\bFFK\otimes_{\bFK}\bM$ is simple.
\end{prop}
\begin{proof}
Recall that $\FK=Q^s$ for some field $Q$ by Lemma \ref{lem:onbffkmodules}(a),
so $\FK(X):=\FFK=Q(X)^s$. Let $\bFK[X]$ be the bold ring
consisting of $\FK[X]=Q[X]^s$ together with the restriction
of $\sigma_{\FFK}$; it acts as the identity on $X$.
Now $\bMcal:=\bFK[X]\otimes_{\bFK}\bM$ is a ``model'' of $\bM'$
in the sense that $\bMcal$ is a restricted $\bFK[X]$-module such that
$\bM'=\bFK(X)\otimes_{\bFK[X]}\bMcal$. Moreover, $\bM=\bMcal/X$.

Assume that $\bM'$ is not simple, so there exists a nontrivial $\bFFK$-submodule $\bN'\subsetneqq\bM'$.
It follows that $\bNcal:=\bMcal\cap\bN'$ is a non-trivial $\bFK[X]$-submodule of $\bMcal$ other than
$\bMcal$, and therefore that $\bN:=\bNcal/(X)$ is a non-trivial $\bFK$-submodule of $\bMcal/(X)\isom\bM$ other than
$\bM$. This contradicts the simplicity of $\bM$, using Proposition \ref{prop:subquotclosedbFK}.
\end{proof}

\begin{proof}[Proof of Theorem \ref{thm:boldscalarext}]
Proposition \ref{prop:boldscalarext1} gives
item (a), and Proposition \ref{prop:boldscalarext2}
gives item (b).
\end{proof}

\section{Translation to Semilinear Algebra}

In this section, we embed the categories of $A$-motives and $A$-isomotives in categories of
bold modules, and classify the categories of integral and rational $\frkp$-adic Galois
representations in terms of categories of bold modules.

This allows us to factor the functors induced by the integral and rational Tate module functors as
composites of two bold scalar extension functors each. The section ends with a proof
that the first factor is ``relatively'' fully faithful in both cases, and semisimple on objects in the rational case.

Let $F,\Fq,A,K,\iota,\sigma_q$ be as in Section $2$. Let $\FK$ denote the total ring of quotients $\Frac(F\otimes_{\Fq}K)$,
it is a field. The bold ring $\bFK$ is given by $\FK$ together with $\sigma=\sigma_{\FK}=\Frac(\id_F\otimes\sigma_q)$.
We refer to Lemma \ref{lem:onbffkmodules} and Proposition \ref{prop:subquotclosedbFK} for the structure of $\bFK$ and its
consquences.
The bold ring $\bAK\subset\bFK$ is given by $\AK:=A\otimes_{\Fq}K$, a Dedekind domain, together with the restriction
$\sigma=\sigma_{\AK}=\id_A\otimes\sigma_q$ of $\sigma_{\FK}$. Given a maximal ideal $\frkp$ of $A$,
let $\AlpK$ denote the subring of $\FK$ consisting of those elements integral at all places $\frkP$ of $\FK$ lying
above $\frkp$, it is a semilocal Dedekind domain. The bold ring $\bAK\subset\bAlpK\subset\bFK$ is given by $\AlpK$ together with the
restriction $\sigma=\sigma_{\AlpK}$ of $\sigma_{\FK}$. We say that an $\bFK$-module $\bM$ is \emph{$\frkp$-restricted}
if it is restricted with respect to the inclusion $\bAlpK\subset\bFK$.

\begin{constr}
An effective $A$-motive $(M,\tau)$ over $K$ induces an $\bFK$-module $\bFK\otimes_{\bAK}(M,\tau)$,
which is $\frkp$-restricted for $\frkp\neq\ker\iota$ by the assumption that $(M,\tau)$ is of
characteristic $\iota$, and hence restricted. Thus the essential image of the tensor functor
$\AMotKeff\to\AKMod\to\FKMod$
consists of dualisable objects by Proposition \ref{prop:subquotclosedbFK}, and so
it extends uniquely to an $A$-linear tensor functor $I_0:\:\AMotK\to\FKMod$.
It maps an $A$-motive $(M,L)$ to $(\bFK\otimes_{\bAK}M)\otimes(\bFK\otimes_{\bAK}L)^\vee$.
Now Theorem \ref{thm:aboutaisomotk}(a) implies that $I_0$ factors through the
category of $A$-isomotives, so there there exists a unique $F$-linear
exact tensor functor $I:\:\AIsomotK\To\FKMod$ such that the following diagram commutes:
\[\xymatrix{
\AMotKeff\ar@{}[d]|{\bigcap} \ar[r]       & \AKMod\ar[d]^{\bFK\otimes_{\bAK}(-)}\\
\AMotK \ar[d]   \ar[r]^{I_0} & \FKMod\\
\AIsomotK \ar[ru]_I
}\]
\end{constr}

\begin{prop}\label{prop:embedaisomotk}
The functor $I$ is fully faithful and semisimple on objects.
For every maximal ideal $\frkp\neq\iota$ of $A$, the essential image
of $I$ consists of $\frkp$-restricted $\bFK$-modules.
\end{prop}
\begin{proof}
The essential image of $I$ consists of $\frkp$-restricted $\bFK$-modules
by construction.

Let us show that $I$ is fully faithful, so let $\bM,\bN$ be
$A$-isomotives. We may assume that both are effective. It is clear that
\[\Hom(\bM,\bN)\to\Hom_{\bFK}(\bFK\otimes_{\bAK}\bM,\bFK\otimes_{\bAK}\bN)\]
is injective, so let $h$ be an element of the target. Now
$h(\bM)$ and  $\bN':=h(\bM)\cap\bN$ are effective $A$-motives, $h|_{\bM}:\:\bM\to h(\bM)$ is
a homomorphism of effective $A$-motives, $h(\bM)\supset\bN'$ is
an isogeny of effective $A$-motives, and $\bN'\subset\bN$ is a homomorphism
of effective $A$-motives:
\[\xymatrix{
\bFK\otimes_{\bAK}\bM \ar[rrr]^h &&& \bFK\otimes_{\bAK}\bN\\
\bM\ar[r]^{h|_{\bM}} \ar@{}[u]|{\cup} & h(\bM) & \bN' \ar@{}[l]|{\supset} \ar@{}[r]|{\subset} & \bN \ar@{}[u]|{\cup}
}\]
Now Proposition \ref{prop:isogfactstand} applied to the isogeny and Theorem \ref{thm:aboutaisomotk}(a)
imply that $h$ is induced by a homomorphism $\bM\to\bN$ of $A$-isomotives.

Let us show that $I$ is semisimple on objects, so let $\bM$ be
a semisimple $A$-isomotive. We may assume that $\bM$ is effective and simple,
since $I$ is additive. Assume that $\bM'_0\subset\bFK\otimes_{\bAK}\bM$ is an $\bFK$-submodule.
Then $\bM':=\bM\cap\bM'_0$ is an effective $A$-isomotive contained in $\bM$, so either
$\bM'=0$ or $\bM'\isom\bM$ by assumption. It follows that $\bFK\otimes_{\bAK}\bM$ is simple.
\end{proof}

We turn to two torsion-free versions of Proposition \ref{prop:torsiongaloisrepsclassified}.
Let $\Ksep$ denote a separable
closure of $K$, with associated Galois group $\GalK:=\Gal(\Ksep/K)$.
Given a maximal ideal $\frkp$ of $A$, let \[\AKp:=\varprojlim_n(A/\frkp^n)\otimes_{\Fq}K\]
denote the completion of $\AK$ at $\frkp$, it is a finite product of pairwise isomorphic
discrete valuation rings. Let $\FKp:=\Frac(\AKp)$ denote the
total ring of quotients of $\AKp$, it is a finite product of pairwise isomorphic fields. The bold ring $\bAKp$ is given by $\AKp$ together
with the endomorphism $\sigma=\sigma_{\AKp}=\varprojlim_n(\id_{A/\frkp^n}\otimes\sigma_q)$ induced by $\sigma_q$,
and the bold ring ring $\bFKp$ is given by $\FKp$ together with the endomorphism
$\sigma=\sigma_{\FKp}=\Frac(\sigma_{\AKp})$ induced by $\sigma_q$.
We say that an $\bFKp$-module $\bM$ is \emph{$\frkp$-restricted}
if it is restricted with respect to the inclusion $\bAKp\subset\bFKp$.

\begin{dfn}
\begin{enumerate}
  \item Let $(M,\tau)$ be restricted $\bAKp$-module.
    We set \[R_{\frkp}'(M,\tau):=(\AKsepp\otimes_{\AKp}M)^{\tau},\]taking $\tau$-invariants
    with respect to the diagonal action. Note that the action of $\GalK$ on $\AKsepp$
    induces an action of $\GalK$ on $R_{\frkp}'(T,\tau)$.
  \item Let $(V,\rho)$ be an integral $\frkp$-adic Galois representation.
    We set \[D_{\frkp}'(V,\rho):=(\AKsepp\otimes_{\Ap}V)^{\GalK},\]taking $\GalK$-invariants
    with respect to the diagonal action. Note that the $\sigma$-linear endomorphism
    of $\AKsepp$ induces a $\sigma$-linear endomorphism $\tau$ of $D_{\frkp}'(V,\rho)$.
\end{enumerate}
\end{dfn}

\begin{prop}\label{prop:intgaloisrepsclassified}
Let $\GalK:=\Gal(\Ksep/K)$ denote the absolute Galois group of $K$.
The functors $D_{\frkp}',R_{\frkp}'$ are quasi-inverse equivalences of $\Ap$-linear rigid
tensor categories:
\[\xymatrix{
\left(\!\!\left(\text{\begin{tabular}{c}integral $\frkp$-adic \\ Galois representations\end{tabular}}\right)\!\!\right) \ar@{}[r]|-{\isom} \ar@<2ex>[r]^-{D_{\frkp}'}&
\left(\!\!\left(\text{\begin{tabular}{c}restricted \\ $\bAKp$-modules\end{tabular}}\right)\!\!\right) \ar@<2ex>[l]^-{R_{\frkp}'}
}\]
Moreover, the following is true:
\begin{enumerate}
  \item $\rk_{\AKp}D_{\frkp}'(V,\rho)=\rk_{\Ap}V$ for every integral $\frkp$-adic Galois representation $(V,\rho)$.
  \item The homomorphism $\AKsepp\otimes_{\Ap}R_{\frkp}'(M,\tau)\to\AKsepp\otimes_{\AKp}M$ is an isomorphism
    for every restricted $\bAKp$-module $(M,\tau)$.
  \item The homomorphism $\AKsepp\otimes_{\AKp}D_{\frkp}'(V,\rho)\to\AKsepp\otimes_{\Ap}V$ is an isomorphism
    for every integral $\frkp$-adic Galois representation $(V,\rho)$.
\end{enumerate}
\end{prop}

\begin{proof}
This follows directly from Proposition \ref{prop:torsiongaloisrepsclassified} by
considering the direct limits involved.
\end{proof}

\begin{dfn}
\begin{enumerate}
  \item Let $(M,\tau)$ be $\frkp$-restricted $\bFKp$-module.
    We set \[R_{\frkp}(M,\tau):=(\FKsepp\otimes_{\FKp}M)^{\tau},\]taking $\tau$-invariants
    with respect to the diagonal action. Note that the action of $\GalK$ on $\FKsepp$
    induces an action of $\GalK$ on $R_{\frkp}(T,\tau)$.
  \item Let $(V,\rho)$ be a rational $\frkp$-adic Galois representation.
    We set \[D_{\frkp}(V,\rho):=(\FKsepp\otimes_{\Fp}V)^{\GalK},\]taking $\GalK$-invariants
    with respect to the diagonal action. Note that the $\sigma$-linear endomorphism
    of $\FKsepp$ induces a $\sigma$-linear endomorphism $\tau$ of $D_{\frkp}(V,\rho)$.
\end{enumerate}
\end{dfn}

\begin{prop}\label{prop:galoisrepsclassified}
Let $\GalK:=\Gal(\Ksep/K)$ denote the absolute Galois group of $K$.
The functors $D_{\frkp},R_{\frkp}$ are quasi-inverse equivalences of $F$-linear rigid abelian
tensor categories:
\[\xymatrix{
\left(\!\!\left(\text{\begin{tabular}{c}rational $\frkp$-adic \\ Galois representations\end{tabular}}\right)\!\!\right) \ar@{}[r]|-{\isom} \ar@<2ex>[r]^-{D_{\frkp}}&
\left(\!\!\left(\text{\begin{tabular}{c}$\frkp$-restricted \\ $\bFKp$-modules\end{tabular}}\right)\!\!\right) \ar@<2ex>[l]^-{R_{\frkp}}
}\]
Moreover, the following is true:
\begin{enumerate}
  \item $\rk_{\FKp}D_{\frkp}(V,\rho)=\dim_{Fp}V$ for every rational $\frkp$-adic Galois representation $(V,\rho)$.
  \item The homomorphism $\FKsepp\otimes_{\Fp}R_{\frkp}(M,\tau)\to\FKsepp\otimes_{\FKp}M$ is an isomorphism
    for every $\frkp$-restricted $\bFKp$-module $(M,\tau)$.
  \item The homomorphism $\FKsepp\otimes_{\FKp}D_{\frkp}(V,\rho)\to\FKsepp\otimes_{\Fp}V$ is an isomorphism
    for every rational $\frkp$-adic Galois representation $(V,\rho)$.
\end{enumerate}
\end{prop}

\begin{proof}
Proposition \ref{prop:intgaloisrepsclassified} implies
this rational version. In fact, ``$D_{\frkp}=D_{\frkp}'$ for rational $\frkp$-adic
Galois representations'' in the sense that
$D_{\frkp}(V,\rho)$ coincides with $(\AKsepp\otimes_{\Ap}V)^{\GalK}$, the definition of $D_{\frkp}'$
applied to $(V,\rho)$, and similarly $R_{\frkp}=R_{\frkp}'$ for $\frkp$-restricted $\bFKp$-modules.
The detailed proof also uses the fact that every rational $\frkp$-adic Galois representation
has a $\GalK$-invariant full $\Ap$-lattice (whereas not every restricted $\bFKp$-module is
$\frkp$-restricted).
\end{proof}

\begin{prop}\label{prop:translatetatemodulefunctor}
For every maximal ideal $\frkp\neq\ker\iota$ of $A$, the following diagram commutes:
\[\xymatrix{
\AIsomotK \ar[rr]^{\Vp}\ar[d]_I && \RepFpGalK \\
\FKMod \ar[rr]_{\bFKp\otimes_{\bFK}(-)} &&\FKpMod \ar[u]^{\Rp}_{\isom}
}\]
\end{prop}
\begin{proof}
This follows directly from the construction of
the categories and functors involved.
\end{proof}

We end this section by applying the main result of Section $3$,
hence proving the ``first half'' of Theorem \ref{thm:mainthm}.
Let $\frkp$ be a maximal ideal of $A$, let $\Fp$ denote
the completion of $F$ at $\frkp$, and set $\bFpK:=\Frac(\Fp\otimes_{\Fq}K,\id\otimes\sigma_q)$.
Note that we have inclusions $\bFK\subset\bFpK\subset\bFKp$,
and that the latter is an equality if and only if $K$ is a finite field.
We set $\bApK:=\bFpK\cap\bAKp$, and say that an $\bFpK$-module is
\emph{$\frkp$-restricted} if it is restricted with respect to the inclusion
$\bApK\subset\bFpK$.

By what we have already proven, Theorem \ref{thm:mainthm} -- the semisimplicity
conjecture -- follows by proving that bold scalar extension $\bFKp\otimes_{\bFK}(-)$
restricted to $\frkp$-restricted $\bFK$-modules is semisimple on objects.
Since
\[\bFKp\otimes_{\bFK}(-)=\big(\bFKp\otimes_{\bFpK}(-)\big)\circ\big(\bFpK\otimes_{\bFK}(-)\big),\]
and being semisimple on objects is a transitive property, we may subdivide
our task in two parts.

\begin{thm}\label{thm:mainthmfirsthalf}
Let $\frkp$ be a maximal ideal of $A$.
Assume that the number of roots of unity of $K$ is finite.
The restriction of the functor of bold scalar extension $\bFpK\otimes_{\bFK}(-)$
to restricted $\bFK$-modules is:
\begin{enumerate}
  \item $\Fp/F$-fully faithful,
  \item semisimple on objects and
  \item maps $\frkp$-restricted modules to $\frkp$-restricted modules.
\end{enumerate}
\end{thm}

\begin{prop}\label{prop:FpFissep}
Every completion $\Fp$ of $F$ at a place $\frkp$ is a separable field extension.
\end{prop}
\begin{proof}
Let us start with the special case of $F=\Fq(t)$ completed at $\frkp=(t)$,
so $\Fp=\ktt$. By \cite[V.\S 15.4]{BourbA} it is sufficient to prove
the following: If $f_1,\ldots,f_m\in\ktt$ are linearly independent
over $\Fq(t)$, then so are the $f_i^p$. Without loss of generality,
assume that $f_i\in \Fq\lleck t\rreck$, and that for certain
$g_i\in k[t]$ we have $\sum_ig_if_i^p=0$. We must show that all $g_i$ are zero.

Since $\Fq$ is perfect, we may write $g_i=:\sum_{j=0}^{p-1}g_{ij}^pt^j$ for
certain $g_{ij}\in k[t]$. These defining equations, together with
$\sum_ig_if_i^p=0$, imply that for all $j$ we have $\sum_ig_{ij}^pf_i^p=0$.
By extracting $p$-th roots of both sides we obtain $\sum_ig_{ij}f_i=0$ for
all $j$. By assumption the $f_i$ are linearly independent, so we
have $g_{ij}=0$ for all $i$ and $j$. Therefore all $g_i$ are zero,
as required.

Let us return to the general setting. We choose a local parameter
$t\in F$ at $\frkp$. Denoting the residue field of $F$ at $\frkp$ by $\kp$,
we have $\Fp=\kp\llkurv t\rrkurv$ and the following commutative diagram
of inclusions:
\[\xymatrix{
\Fq(t) \ar[r] \ar[d] & F \ar[d]\\
\ktt \ar[r]& \kp\llkurv t\rrkurv
}\]
We have just seen that $\Fq(t)\subset\ktt$ is separable; clearly, so is
$\ktt\subset\kp\llkurv t\rrkurv$, hence $\Fq(t)\subset\Fp$ is separable.
Moreover, $\Fq(t)\subset F$ is separable algebraic since $t$ is a local parameter.
This implies that $F\subset\Fp$ is separable by \cite[V.\S15]{BourbA}.
\end{proof}

\begin{proof}[Proof of Theorem \ref{thm:mainthmfirsthalf}]
Since $\Fp/F$ is separable by Proposition \ref{prop:FpFissep},
Theorem \ref{thm:boldscalarext} implies items (a) and (b)
of Theorem \ref{thm:mainthmfirsthalf}. Item (c) follows
from the fact that $\AlpK=\FK\Cap\AKp$.
\end{proof}

\section{Tamagawa-Fontaine Theory}

In this section, we complete the proof of Theorem
\ref{thm:mainthm} with the help of what we term
``Tamagawa-Fontaine theory'', since the basic
ideas a sketch of the proofs are due to Tamagawa \cite{Tam95}
and have some formal analogy to Fontaine theory.

Let $F,\Fq,A,\frkp$ be as before,
let $K/\Fq$ be a \begin{tabular}{|c|}
\hline \emph{finitely generated}\\\hline\end{tabular}
field and let $\Ksep$ denote a separable closure of $K$
with associated absolute Galois group $\GalK:=\Gal(\Ksep/K)$.
Recall that we have constructed bold rings $\bApK\subset\bFpK$
and $\bAKp\subset\bFKp$, and that we call $\bFpK$- and $\bFKp$-modules
\emph{$\frkp$-restricted} if they are restricted with respect to these
inclusions. To any $\frkp$-restricted $\bFpK$-module $\bM$, we associate
the rational $\frkp$-adic Galois representation
\[\Vp(\bM):=\Rp(\bFKp\otimes_{\bFpK}\bM).\]
\begin{dfn}
Following \cite{Tam95}, we say that a rational $\frkp$-adic Galois representation
is \emph{quasigeometric} if it is isomorphic to $\Vp(\bM)$ for some
$\frkp$-restricted $\bFpK$-module $\bM$.
\end{dfn}
The theory consists of constructing
a bold ring $\bB\subset\bFKsepp$ and developing the properties
of the associated functor\footnote{``C'' for coreflection.}
\[\Cp:=\big((\bB\otimes_{\Fp}(-)\big)^{\GalK}:\:\RepFpGalK\to\FpKModres.\]
It allows to determine which rational $\frkp$-adic
Galois representations are quasigeometric
(those for which $\rk_{\FpK}\big(\Cp (V,\rho)\big)=\dim_{\Fp}V$),
and its properties imply that $\bFKp\otimes_{\bFpK}(-)$, restricted
to $\frkp$-restricted $\bFpK$-modules, is fully faithful and semisimple
on objects. Thereby, the proof of Theorem \ref{thm:mainthm} is completed.

We choose to postpone the construction of $\bB$ to the next section
(Definitions \ref{dfn:TamagawasBplusS} and \ref{dfn:TamagawasB}),
and develop the properties of $\Cp$ using only the properties of $\bB$
given in the following claim. These properties will also
be established in the next section (Theorem \ref{thm:Bexists}).

\begin{claim}\label{claim:Bexists}
Assume that $K/\Fq$ is finitely generated.
There exists a ring $B\subset\FKsepp$ with the following
properties:
\begin{enumerate}
  \item $\sigma_{\FKsepp}(B)\subset B$ and $B^\sigma=\Fp$.
  \item $\GalK(B)\subset B$ and $B^{\GalK}=\FpK$.
  \item Every $\frkp$-restricted $\bFpK$-module $\bM$ fulfills $\Vp(\bM)\subset\bB\otimes_{\bFpK}\bM$.
\end{enumerate}
\end{claim}

Note that the existence of such a ring of periods is a matter of \emph{construction},
since property (b) requires $\bB$ to be ``small enough'' (as $(\FKsepp)^\GalK=\FKp$ strictly
contains $\FpK$ if $K$ is infinite), whereas property (c) requires $\bB$ to be ``large enough''
(as it must contain the Galois-invariant elements of $\bFKsepp\otimes_{\bFpK}\bM$ for every
$\frkp$-restricted $\bFpK$-module $\bM$).

This claim will be justified in Theorem \ref{thm:Bexists}.
Until the end of the proof of Theorem \ref{thm:tfmainthm},
\emph{we will assume that Claim \ref{claim:Bexists} is true}.
Note that there exists a smallest ring with the properties
required in Claim \ref{claim:Bexists}, the intersection of 
the (non-empty!) set of such rings. What follows does not depend on
our choice of $\bB$. But we might as well choose this canonical
smallest $\bB$ in the following, so we do.

\begin{lem}\label{lem:descendequivisom}
Let $\bM=(M,\tau)$ be a $\frkp$-restricted $\FpK$-module. Then the natural comparison isomorphism
$\FKsepp\otimes_{\Fp}\Vp(\bM)\to\FKsepp\otimes_{\FpK}M$ of Proposition
\ref{prop:galoisrepsclassified}(b) descends to a $\GalK$-equivariant isomorphism
of $\bB$-modules
\[c_{\bM}:\:\bB\otimes_{\Fp}\Vp(\bM)\To\bB\otimes_{\bFpK}\bM\]
\end{lem}
\begin{proof}
Claim \ref{claim:Bexists}(b) implies that the given isomorphism
descends to a $\GalK$-equivariant homomorphism
of $\bB$-modules \[c_{\bM}:\:\bB\otimes_{\Fp}\Vp(\bM)\To\bB\otimes_{\bFpK}\bM.\]
Since both sides are free $B$-modules of finite rank, it suffices to show
that the determinant of $c_{\bM}$ is an isomorphism. Since $\Vp$ is a tensor
functor, we have
\[\det\left(c_{\bM}\right)=c_{\det(\bM)},\]
so we may reduce to the case where $\rk(\bM)=1$. In this case,
choosing a basis for both $\Vp(\bM)$ and $\bM$, we see that $c_{\bM}$ is given by left multiplication by
an element $c(\bM)\in B$. Choosing the dual bases of $\Vp(\bM^\vee)$ and $\bM^\vee$,
analogously $c_{\bM^\vee}$ is given by left multiplication by an element $c(\bM^\vee)\in B$.

By Proposition \ref{prop:galoisrepsclassified}(c), the element $c(\bM)$
is invertible in $\FKsepp$. By naturality, its
inverse $c(\bM)^{-1}$ coincides with $c(\bM^\vee)$. Since both
$c(\bM)$ and $c(\bM^\vee)$ lie in $B$, $c_{\bM}$ is indeed
an isomorphism.
\end{proof}

We continue to exploit the consequences of Claim \ref{claim:Bexists}.

\begin{thm}\label{thm:Vpff}
The functor $\Vp$ on $\frkp$-restricted $\FpK$-modules is fully faithful.
\end{thm}
\begin{proof}
Consider two $\frkp$-restricted $\bFpK$-modules $\bM,\bN$. By Lemma \ref{lem:descendequivisom}
we have a $\tau$- and $\GalK$-equivariant natural isomorphism
\[\bB\otimes\bM^\vee\otimes\bN\To\bB\otimes\Vp(\bM^\vee\otimes\bN)=\bB\otimes\Vp(\bM)^\vee\otimes\Vp(\bN),\]
which implies that
\[(\bB\otimes\bM^\vee\otimes\bN)^{\Gamma,\tau}\isom(\bB\otimes\Vp(\bM)^\vee\otimes\Vp(\bN))^{\tau,\Gamma}\]
Since $\Hom(\bM,\bN)=(\bM^\vee\otimes\bN)^\tau$ coincides with the domain
of this isomorphism, and $\Hom(\Vp(\bM)^\vee,\Vp(\bN))=(\Vp(\bM)^\vee\otimes\Vp(\bN))^\GalK$
coincides with its target, we see that $\Vp$ is indeed fully faithful.
\end{proof}

\begin{dfn}\begin{enumerate}
  \item Let $(V,\rho)$ be a rational $\frkp$-adic Galois representation. We set
     \[\Cp(V,\rho):=(B\otimes_{\Fp}V)^{\GalK},\]taking Galois-invariants
     with respect to the diagonal action. Note that the $\sigma$-linear endomorphism
     of $B$ induces a $\sigma$-linear endomorphism $\tau$ of $\Cp(V,\rho)$.
  \item Set $\bOB:=\bB\cap\bAKsepp$. Let $(T,\rho)$ be an integral $\frkp$-adic Galois
    representation. We set
    \[\OCp(T,\rho):=(\OB\otimes_{\Ap}T)^{\GalK},\]taking Galois-invariants
    with respect to the diagonal action. Note that the $\sigma$-linear
    endomorphism of $\OB$ induces a $\sigma$-linear endomorphism $\tau$ of $\OCp(T,\rho)$.
\end{enumerate}
\end{dfn}

\begin{lem}\label{lem:corofffff}
For every $\frkp$-restricted $\bFpK$-module $\bM$, the comparison
isomorphism $c_{\bM}$ of Lemma \ref{lem:descendequivisom} induces an isomorphism
of $\bFpK$-modules
\[\Cp(\Vp\bM)\arrover{\isom}\bM.\]
\end{lem}
\begin{proof}
Take $\GalK$-invariants!
\end{proof}

\begin{prop}\label{prop:Qptens}
\begin{enumerate}
  \item $\Cp$ is an exact $\Fp$-linear tensor functor.
  \item $\OCp$ is an exact $\Ap$-linear tensor functor.
\end{enumerate}
\end{prop}
\begin{proof}
$\OCp$ and $\Cp$ are left exact linear functors by definition.
Let us show that they are tensor functors. We will deduce
this from the fact that the functors $D_{\frkp}'$ and $D_{\frkp}$ of Section $3$ are such.

Let us do this for $\Qp$, mutatis mutandis the proof is the same for $\OQp$.
Consider a rational $\frkp$-adic Galois representation $\bV=(V,\rho)$.
We have $D_{\frkp}(\bV)=(\FKsepp\otimes_{\Fp}V)^\GalK$ and $\Cp(\bV)=(B\otimes_{\Fp}V)^{\GalK}$.
Therefore, calculating in
$\FKsepp\otimes_{\Fp}V$, we have $\Qp(\bV)=(B\otimes_{\Fp}V)\cap D_{\frkp}(\bV)$.

Given another rational $\frkp$-adic Galois representation $\bW$, we may apply these remarks to $\bV$, $\bW$
and $\bV\otimes_{\Fp}\bW$. In $\FKsepp\otimes_{\Fp}V\otimes_{\Fp}W$ we calculate:
\begin{eqnarray*}
\Qp(V\otimes_{\Fp} W) &=& (B\otimes_{\Fp}V\otimes_{\Fp}W)\cap D_{\Fp}(V\otimes_{\Fp} W)\\
								&=& \left((B\otimes_{\Fp}V)\otimes_B(B\otimes_{\Fp}W)\right) \cap \left(D_{\Fp}(V)\otimes_{\FKp}D_{\Fp}(W)\right)\\
								&=& \left((B\otimes_{\Fp}V)\cap D_{\Fp}(V)\right) \otimes_{\FpK} \left((B\otimes_{\Fp}W)\cap D_{\Fp}(W)\right)\\
								&=& \Qp(V)\otimes_{\FpK}\Qp(W).
\end{eqnarray*}

Finally, the right exactness of $\Qp$ and $\OQp$ follows formally from what we have
proven. Again, we do this only for $\Qp$, mutatis mutandis the proof is the same for $\OQp$.
Since $\Qp$ is a tensor functor and $\bV$ admits a dual $\bV^\vee$, the $\bFpK$-module
$\Qp(\bV)$ has a dual, namely $\Qp(\bV^\vee)$. Therefore, if
\[\bV'\to\bV\to\bV''\to 0\]
is a right exact sequence of rational $\frkp$-adic Galois representations, then its image under $\Qp$ coincides with the
dual of the image of the left exact sequence $0\to(\bV'')^\vee\to\bV^\vee\to(\bV')^\vee$. Since
$\Qp$ is left exact, the image of this left exact sequence is left exact. Since dualisation is exact,
the image of our original right exact sequence is right exact, and we are done.
\end{proof}

\begin{lem}\label{lem:onbfkpmodules}
\begin{enumerate}
  \item The ring $\FKp$ is a finite product of fields, each
    isomorphic to a field of Laurent series $K'\llkurv t\rrkurv$ for some finite
    extension $K'/K$.
  \item The underlying $\FKp$-module of every restricted $\bFKp$-module is free.
\end{enumerate}
\end{lem}
\begin{proof}
Let $t\in A$ denote a local parameter at $\frkp$,
and let $\bbF_\frkp$ denote the residue field of $\frkp$.
By definition, $\FKp=\Frac(\AKp)$, and we have
\[\AKp=\varprojlim_n\big((A/\frkp^n)\otimes_{\Fq}K\big)=\varprojlim_n\big((\bbF_\frkp[t]/t^n)\otimes_{\Fq}K)=\big(\bbF_\frkp\otimes_{\Fq} K)\lleck t\rreck.\]
As in the proof of Lemma \ref{lem:onbffkmodules}(a), $\bbF_\frkp\otimes_{\Fq} K\isom (K')^{\times s}$
for some finite field extension $K'/K$ and integer $s\ge 1$. It follows
that $\FKp=K'\llkurv t\rrkurv^{\times s}$ has the property stated in item (a).
Item (b) follows as in the proof of Lemma \ref{lem:onbffkmodules}(b).
\end{proof}

\begin{lem}\label{lem:OBisproj}
\begin{enumerate}
   \item $\OB$ is a projective $\ApK$-module.
   \item $B$ is a projective $\FpK$-module.
\end{enumerate}
\end{lem}
\begin{proof}
By Lemma \ref{lem:onbffkmodules}, $\FpK=Q_1\times\cdots\times Q_s$ is a finite
product of fields. Setting $B_i:=Q_i\otimes_{\FpK}B$, wo obtain
a decomposition $B=B_1\times\cdots\times B_s$. Since the $Q_i$ are fields,
the $B_i$ are free $Q_i$-modules, so $B$ is a projective $\FpK$-module.

To show that this implies that $\OB$ is a projective $\ApK$-module, we
need some notation. Choose a local parameter $t\in F$ at $\frkp$.
We have $\FpK\subset\FKp$, and the latter ring splits as $\FKp=Q_1'\times\cdots\times Q_s'$
with $Q_i'\isom K'((t))$ for a finite field extension $K'\supset K$ by
Lemma \ref{lem:onbfkpmodules}. We may thus identify the fields $Q_i$ with subfields
of $Q_i'=K'\llkurv t\rrkurv$, note that $Q_i$ \emph{contains} $t$.

Under this identification, setting $R_i:=Q_i\cap K_r\llkurv t\rrkurv$,
we have $\ApK=R_1\times\cdots\times R_s$.

The ring $B$ is a subring of
\[\FKsepp\isom(\kp\otimes_k\Ksep)\llkurv t\rrkurv=(\kp\otimes_kK\otimes_K\Ksep)\llkurv t\rrkurv=(K'\otimes_K\Ksep)\llkurv t\rrkurv^{\times s},\]
with $B_i$ contained in the $i$-th copy of $(K_r\otimes_K\Ksep)\llkurv t\rrkurv$.
The ring $\OB$ splits as $\OB_1\times\cdots\times\OB_s$, where $\OB_i:=\OB\cap B_i$ is the ring consisting
of those elements of $B_i$ which, viewed as elements of the $i$-th copy of $(K_r\otimes_K\Ksep)\llkurv t\rrkurv$ in $\FKsepp$,
are power series, that is, lie in $(K_r\otimes_K\Ksep)\lleck t\rreck$.

Let us show that $\OB_i$ is a free $R_i$-module, which implies that $B'$ is a projective $\ApK$-module.
For this, we choose a $Q_i$-basis $\{b_{ij}\}_{j\in J_i}$ of $B_i$. Under the identifications given above,
each $b_{ij}$ corresponds to a Laurent series $\sum b_{ijn} t^n$ in $(K_r\otimes_K\Ksep)\llkurv t\rrkurv$.
Now $K_r\otimes_K\Ksep\isom(\Ksep)^{\times r}$ for some $r\ge 1$, whereby $1\otimes 1$
corresponds to an element $(e_1,\ldots,e_\rho)$. By multiplying $b_{ij}$ with a suitable element
of the form $(e_1t^{n(i,j,1)},\ldots,e_\rho t^{n(i,j,r)})$, we may assume that $b_{ijn}=0$ for $n<0$
and that $b_{ij0}$ is invertible in $K_r\otimes_K\Ksep$. And then under this assumption, one may check that $\{b_{ij}\}$
is indeed an $R_i$-basis of $\OB_i$.
\end{proof}

\begin{lem}\label{lem:aninjmap}
\begin{enumerate}
  \item The natural homomorphism $\AKp\otimes_{\ApK}B'\To\AKsepp$ is injective.
  \item The natural homomorphism $\FKp\otimes_{\FpK} B\To\FKsepp$ is injective.
\end{enumerate}
\end{lem}
\begin{proof}
(a): We will use the following facts from commutative algebra: Given an ideal $I\subset R$
of a commutative ring $R$ such that $\bigcap I^n=0$, let $\widehat{R}:=\varprojlim_nR/I^n$
denote the $I$-adic completion of $R$. If $M$ is a projective $R$-module,
then the natural homomorphism
$\widehat{R}\otimes_RM\to\widehat{M}:=\varprojlim_n M/(I^n\cdot M)$ is injective.
If $M\to N$ is an injective homomorphism of $R$-modules, then the induced
homomorphism $\widehat{M}\to\widehat{N}$ is injective.

The first of these facts is checked easily for free $R$-modules,
and this implies the statement for projective $R$-modules by the additivity
of source and target. The second fact is a consequence of the left exactness
of $\varprojlim$.

By Lemma \ref{lem:OBisproj}, we may apply this to $R=\ApK$, $I=\frkp$, $M=B'$
and $N=\FKsepp$, and obtain the desired injectivity of
\[\AKp\otimes_{\ApK}B'\to\widehat{B'}\to\widehat{\FKsepp}=\FKsepp.\]

(b): This item follows from item (a)
by inverting any local parameter $t\in F$ at $\frkp$.
\end{proof}

\begin{prop}\label{cor:aninjmap}
\begin{enumerate}
  \item For every integral $\frkp$-adic representation $\bT$, the following natural map is injective:
           \[\AKp\otimes_{\ApK}\OQp(\bT)\To D_{\frkp}'(\bT)\]
  \item For every rational $\frkp$-adic representation $\bV$, the following natural map is injective:
           \[\FKp\otimes_{\FpK}\TFp(\bV)\To D_{\frkp}(\bV)\]
\end{enumerate}
\end{prop}
\begin{proof}
(a): We calculate:
\begin{eqnarray*}
\AKp\otimes_{\ApK}\OQp(\bT) &=& \AKp\otimes_{\ApK}\left(B'\otimes_{\Ap}T\right)^{\GalK}\\
												  	&=& \left(\AKp\otimes_{\ApK}B'\otimes_{\Ap}T\right)^{\GalK}\\
											&\subset &\left(\AKsepp\otimes_{\Ap} T\right)^{\GalK}\qquad\text{by Lemma \ref{lem:aninjmap}(a)}\\
												      &=& D_{\frkp}'(\bT),
\end{eqnarray*}
(b): We may repeat the calculation of (a), using Lemma \ref{lem:aninjmap}(b).
\end{proof}

\begin{prop}\label{prop:Qpfunctor}
\begin{enumerate}
  \item The functor $\OQp$ has values in restricted $\bApK$-modules.
  \item The functor $\Qp$ has values in $\frkp$-restricted $\bFpK$-modules.
  \item For every rational $\frkp$-adic Galois representation $\bV=(V,\rho)$ we have $\rk_{\FpK}\Qp(\bV)\le\dim_{\Fp}V$.
\end{enumerate}
\end{prop}

\begin{proof}
For every rational representation $\bV$ there exists an integral representation
$\bT=(T,\rho)$ such that $\bV=\Fp\otimes_{\Ap}\bT$, and then $\Qp(\bV)=\bFpK\otimes_{\bApK}\OQp(\bT)$.
Therefore, it suffices to show that $\OQp(\bT)$ is a restricted $\bApK$-module
of rank bounded above by $\rk_{\Ap}T$.

By Proposition \ref{cor:aninjmap}(a), $\AKp\otimes_{\ApK}\OQp(\bT)$ is a submodule
of $D_{\frkp}'(\bT)$, which is a free $\AKp$-module of rank $\rk_{\Ap}T$. Therefore, $\OQp(\bT)$ is a finitely generated
projective $\ApK$-module.
Since $D_{\frkp}'(\bT)$ has a bijective $\taulin$,
its submodule $\AKp\otimes_{\ApK}\OQp(\bT)$ has an injective $\taulin$, and therefore
$\OQp(\bT)$ has an injective $\taulin$ as well. Since $\sigma_{\ApK}$ permutes the factors
of $\ApK=\FpK\cap\ApK\isom (Q_1\times\cdots Q_s)\cap\ApK=(Q_1\cap\ApK)\times\cdots\times(Q_s\cap\ApK)$
the injectivity of $\taulin$ implies that
$\OQp(\bT)$ is free of \emph{constant} rank $r:=\rk_{\ApK}\OQp(\bT)\le\rk_{\AKp}\bT$,
as in the proof of Lemma \ref{lem:onbffkmodules}(b).

It remains to show that the $\taulin$ of $\OQp(\bT)$ is bijective.
Clearly, this is the case if and only if the $\taulin$ of the determinant
of $\OQp(\bT)$ is bijective. By Proposition \ref{prop:Qptens}(a),
$\OQp$ is a tensor functor, so we obtain an inclusion
\[\AKp\otimes_{\ApK}\OQp\left(\bigwedge_{\Ap}^r\bT\right)\subset D_{\frkp}'\left(\bigwedge_{\Ap}^r\bT\right),\]
where the right hand side is a restricted $\bAKp$-module of rank $\ge 1$.
Tracing through the definitions, we see that the left hand side is saturated in the
right hand side, i.e., the quotient is a projective $\AKp$-module.
An application of the Snake Lemma shows that this implies
that $\AKp\otimes_{\ApK}\OQp(\Lambda^r\bT)$ has bijective $\taulin$.
Now the equality $\ApK^\times=\AKp^\times\cap\ApK$ implies that $\OQp(\bT)$ itself has
bijective $\taulin$.
\end{proof}

\begin{prop}\label{prop:inequality}
Let $\bV=(V,\rho)$ be a rational $\frkp$-adic Galois representation.
\begin{enumerate}
  \item $\bV$ is quasigeometric if and only if $\rk_{\FpK}\TFp(\bV)=\rk_{\Fp}V$.
  \item $\Vp(\Qp(\bV))$ is the largest quasigeometric subrepresentation of $\bV$.
  \item If $\bV$ is quasigeometric, then so is every subquotient of $\bV$.
\end{enumerate}
\end{prop}
\begin{proof}
(a): Assume that $\bV\isom\Vp(\bM)$ is quasigeometric.
By Lemma \ref{lem:corofffff}, $\Qp(\Vp(\bM))\isom\bM$.
Therefore, using the fact that $\Vp$ preserves ranks,
we have\[\rk\Qp(\bV)=\rk\TFp\left(\Vp(\bM)\right)=\rk(\bM)=\rk\Vp(\bM)=\rk\bV,\]as claimed.

Assume that we have an equality of ranks. By Proposition \ref{cor:aninjmap}(b),
the natural homomorphism $\bFKp\otimes_{\bFpK}\TFp(\bV)\To D_{\frkp}(\bV)$ is injective.
Since $D_{\frkp}$ preserves ranks, both sides are free of
equal finite rank over the semisimple commutative ring $\FKp$. So the homomorphism is an isomorphism!
We set $\bM:=\TFp(\bV)$, a $\frkp$-restricted $\bFpK$-module by Proposition \ref{prop:Qpfunctor}.
Then the following isomorphisms shows that $\bV$ is quasigeometric:
\[\bV\isom R_{\frkp}(D_{\frkp}(\bV))\isom R_{\frkp}(\bFKp\otimes_{\bFpK}\Qp(\bV))=\Vp(\Qp(\bV))=\Vp(\bM).\]

(b): $\Vp(\Qp\bV)$ is quasigeometric by Proposition \ref{prop:Qpfunctor}(b). Proposition
\ref{cor:aninjmap}(b) and the exactness of $\Vp$ imply that $\Vp(\Qp\bV)$ is a subrepresentation
of $\bV$. Let us show that it contains every other quasigeometric subrepresentation $\Vp(\bM')\isom \bV'\subset \bV$.
By restricting the isomorphism $c_{\bM'}$ of Lemma \ref{lem:descendequivisom} to $\GalK$-invariants,
we have $\bM'=\Qp(\Vp\bM')$.  So using the left-exactness of $\Qp$, we see
that
\[\bM'=\Qp(\Vp\bM')=\Qp\bV'\subset\Qp\bV.\]
In turn, since $\Vp$ is exact, this shows that $\bV'=\Vp(\bM')\subset\Vp(\Qp\bV)$, as claimed.

(c): Let $0\to\bV'\to\bV\to\bV''\to 0$ be an exact sequence of representations, and assume that $\bV$
is quasigeometric. Consider the induced sequence
\begin{equation}\label{eqn:inducse}
0\To\Qp\bV'\To\Qp\bV\To\Qp\bV'\To 0
\end{equation}

It is exact by Proposition \ref{prop:Qptens}. Applying the exact functor $\Vp$,
we obtain an exact sequence
\[0\To\Vp\Qp\bV'\To\bV\To\Vp\Qp\bV''\To 0,\]
where $\bV=\Vp\Qp\bV$ by item (b).
Now\[\rk\bV=\rk\Vp\Qp\bV'+\rk\Vp\Qp\bV''\le\rk\bV'+\rk\bV''=\rk\bV\]implies
that $\rk\Vp\Qp\bV'=\rk\bV'$ and $\rk\Vp\Qp\bV'=\rk\bV'$, so $\bV'=\Vp\Qp\bV'$
and  $\bV''=\Vp\Qp\bV''$ are both quasigeometric by (a).
\end{proof}

We collect our results in a categorical reformulation.

\begin{thm}\label{thm:tfmainthm}
\begin{enumerate}
  \item $\Vp:\:\FpKModres\to\RepFpGalK$ is an exact $\Fp$-linear tensor functor
    which is fully faithful and semisimple on objects.
  \item The pair $(\Vp,\Qp)$ is an adjoint pair of functors, that is, for every $\frkp$-restricted $\bFpK$-module $\bM$
  and rational $\frkp$-adic Galois representation $\bV$ there exists a natural isomorphism of $\Fp$-vector spaces
     \[\Hom\big(\Vp(\bM),\bV\big)\To\Hom\big(\bM,\Qp(\bV)\big)\]
  \item The unit $\id\Rightarrow\Qp\circ\Vp$ of this adjunction is an isomorphism (so $\Qp$ is a ``coreflection''
     of the ``inclusion'' $\Vp$).
  \item The counit $\Vp\circ\Qp\Rightarrow\id$ of this adjunction is a monomorphism.
\end{enumerate}
\end{thm}
\begin{proof}
(a): $\Vp=\Rp\circ\big(\bFKp\otimes_{\bFpK}(-)\big)$ is an exact $\Fp$-linear tensor functor
as a composition of such.
It is fully faithful by Theorem \ref{thm:Vpff}. Proposition \ref{prop:inequality}(c) implies that $\Vp$
maps simple objects to simple objects, so it is semisimple on objects.

(b): Let us construct the inverse of the adjunction isomorphism for a given $\bM$ and $\bV$.
Since $\Vp$ is fully faithful, we have a natural isomorphism
\[\Vp:\:\Hom(\bM,\Qp\bV)\To\Hom(\Vp\bM,\Vp\Qp\bV)\]
One the other hand, every homorphism $\Vp\bM\to\bV$ has a quasigeometric image by
Proposition \ref{prop:inequality}(c), which must lie in $\Vp\Qp\bV$ by Proposition \ref{prop:inequality}(b).
Therefore, $\Hom(\Vp\bM,\Vp\Qp\bV)=\Hom(\Vp\bM,\bV)$, and we are done.

(c,d): Both items follow from Proposition \ref{prop:inequality}.
\end{proof}

\begin{proof}[Proof of Theorem \ref{thm:mainthm}]
By Proposition \ref{prop:translatetatemodulefunctor}, the functor
$\Vp$ on $A$-isomotives coincides with $\Rp\circ\big(\FKp\otimes_{\FK}(-)\big)\circ I$.
By Proposition \ref{prop:embedaisomotk}, $I$ is semisimple on objects
and has $\frkp$-restricted values.
The functor $\big(\FKp\otimes_{\FK}(-)\big)$ is a composition
of the functors $\big(\FpK\otimes_{\FK}(-)\big)$ and $\big(\FKp\otimes_{\FpK}(-)\big)$.
The former is semisimple on restricted $\bFK$-modules and maps
$\frkp$-restricted modules to $\frkp$-restricted $\bFpK$-modules
by Theorem \ref{thm:mainthmfirsthalf}(b,c), whereas the latter is
semisimple on $\frkp$-restricted $\bFpK$-modules by Theorem \ref{thm:tfmainthm}(a).
The functor $\Rp$ is semisimple on $\frkp$-restricted $\bFKp$-modules
since it is an equivalence of categories. Therefore,
$\Vp$ is semisimple on objects, being a composition of such functors.
\end{proof}

We end this section with a proof of the Tate conjecture for
$A$-motives.

\begin{prop}\label{prop:tateconj}
Let $K$ be a field which is finitely generated over a finite field.
Let $\frkp\neq\ker\iota$ be a maximal ideal of $A$.
\begin{enumerate}
  \item Let $M,N$ be $A$-motives of characteristic $\iota$. The natural
  homomorphism $\Ap\otimes_A\Hom(M,N)\to\Hom(\Tp M,\Tp N)$ is an isomorphism.
  \item Let $X,Y$ be $A$-isomotives of characteristic $\iota$. The natural
  homomorphism $\Fp\otimes_F\Hom(X,Y)\to\Hom(\Vp M,\Vp N)$ is an isomorphism.
\end{enumerate}
\end{prop}
\begin{proof}
(b): By Proposition \ref{prop:translatetatemodulefunctor} we have
$\Vp=\Rp\circ\big(\FKp\otimes_{\FK}(-)\big)\circ I$. As in the proof
of Theorem \ref{thm:mainthm}, this implies reduces the
proof that $\Vp$ is $\Fp/F$-fully faithful to Proposition \ref{prop:embedaisomotk},
Theorem \ref{thm:mainthmfirsthalf}(a,c) and Theorem \ref{thm:tfmainthm}(a).

(a): The image of the given homomorphism
$\Ap\otimes_A\Hom(M,N)\to\Hom(\Tp M,\Tp N)$ is saturated -- this is well-known.
Therefore, item (b) implies item (a).
\end{proof}

\section{Constructing a Ring of Periods}

We turn to the laborious task of constructing a ring
$B$ which fulfills Claim \ref{claim:Bexists}.

Recall that we assume that $K$ is a \emph{finitely generated} field extension
of a finite field $\Fq$ with $q$ elements. We identify $K$ with the function field $\Fq(X)$
of a proper normal variety $X$ over $\Fq$. For every finite Galois extension
$\Ksep\supset L\supset K$, let $X_L$ be the normalisation of $X$ in $L$,
this is a proper normal variety over $L$.

Let $\Sigma_L$ be the set of prime (Weil) divisors of $X_L$. For every
Galois tower \[\Ksep\supset L'\supset L\supset K\]we have a
projection map $\pr_{L,L'}:\:\Sigma_{L'}\To\Sigma_L$, so we may
let
\[\Ssep:=\varprojlim_{L\supset K}\Sigma_L\]
be the projective limit along the projections $\pr_{L',L}$. Given
a Galois extension $L\supset K$, an element $x_L\in\Sigma_L$ and an 
element $x\in\Ssep$, we say that \emph{$x$ lies over $x_L$} if
$x_L$ is the $L$-th component of $x$.

For each
$x=(x_L)_L\in\Ssep$, there is a unique associated valuation
\[v_x:\:\Ksep\To\bbQ\cup\{\infty\}\]
extending the normalised valuation $v_{x_K}$ of $K$ associated to $x_K$.
Explicitly, for $f\in\Ksep$ we may choose a finite Galois extension
$K\subset L\subset\Ksep$ containing $f$, and set $v_x(f):=v_{x_L}(f)/e_{x_L}$,
where $v_{x_L}$ denotes the normalised valuation of $L$ associated to
$x_L$, and $e_{x_L}$ is the index of $v_{x_L}(K^*)$ in $v_{x_L}(L^*)$.

Let $F$ be a global field with field of constants $\Fq$, and fix
a place $\frkp$ of degree $d:=\deg\frkp$ of $F$ with residue field $\kp$.
We wish to extend $v_x$ to a function
on $\FKsepp$. For calculational reasons, we choose a local parameter $t\in F$ at $\frkp$
and obtain identifications $\OFKsepp=(\kp\otimes_k\Ksep)\lleck t\rreck$
and $\FKsepp=(\kp\otimes_k\Ksep)\llkurv t\rrkurv=\OFKsepp[t^{-1}]$. Recall that by
Lemma \ref{lem:onbfkpmodules} the homomorphism
\begin{equation}\label{eqn:kpotiKsep}
(\kp\otimes_k\Ksep,\id\otimes\sigma)\To\left((\Ksep)^{\times d},\sigma'\right)
\end{equation}
mapping $x\otimes y$ to $(x\cdot\sigma_q^i(y))_{i=0}^{d-1}$ is an isomorphism
of bold rings, with
\[\sigma'(z_0,\ldots,z_{d-1})=(z_{d-1}^q,z_0^q,\ldots,z_{d-2}^q)\]
for $(z_0,\ldots,z_{d-1})\in(\Ksep)^{\times d}$.
We will denote the action of $\sigma'$ on $(\Ksep)^{\times d}$ simply by $\sigma$.
Writing an element $f\in\FKsepp$ as $f=\sum_{i\gg -\infty}f_it^i$ with
$f_i=(f_{ij})_j\in(\Ksep)^{\times d}$, we set
\[v_x(f):=\inf_i\min_jv_x(f_{ij}).\]
Moreover, for all $m,n\ge 1$ and $\Delta=(\delta_{ij})\in\Mat_{m\times n}(\FKsepp)$
we set
\[v_x(\Delta):=\inf_{i,j}v_x(\delta_{ij}).\]

\begin{prop}\label{prop:vx}
For each $x\in\Ssep$ and all $m,n\ge 1$, the function
\[v_x:\:\Mat_{m\times n}(\FKsepp)\To\bbR\cup\{\pm\infty\}\]
is well-defined and independent of the choices made.
For $m=n=1$ and all $f,g\in\FKsepp$ it has the following properties:
\begin{enumerate}
  \item $v_x(f+g)\ge\min\{v_x(f),v_x(g)\}$.
  \item $v_x(fg)\ge v_x(f)+v_x(g)$ (using the convention $-\infty+\infty=-\infty$).
  \item $v_x(\sigma(f))=q\cdot v_x(f)$.
\end{enumerate}
\end{prop}
\begin{proof}
Since $v_x(\kp^\times)=0$, the choice of local parameter does not influence the
definition of $v_x$. Now (a,b) follow from short calculations
using the semicontinuity of infima, whereas (c) follows from 
(\ref{eqn:kpotiKsep}).
\end{proof}

\begin{rem}
Note that, in general, we do not have $v_x(fg)=v_x(f)+v_x(g)$.
\end{rem}

\begin{prop}\label{cor:fixpt}
For all integers $m,n\ge 1$, matrices $\Delta\in\Mat_{n\times n}(\FKsepp)$
and column vectors $F\in{\FKsepp}^{\oplus n}$ the equation $\sigma^m(F)=\Delta\cdot F$ implies
the inequality \[v_x(F)\ge\frac{1}{q^m-1}v_x(\Delta).\]
\end{prop}
\begin{proof}
If $v_x(\Delta)=-\infty$, the inequality stated is tautological,
so we assume that $C:=v_x(\Delta)\neq-\infty$. By a matrix-version of Proposition \ref{prop:vx},
the equation $\sigma^m(F)=\Delta F$ would imply that $q^m\cdot v_x(F)\ge C+v_x(F)$.
If also $v_x(F)\neq\pm\infty$, this would imply the claim of this Proposition.
However, if $v_x(F)=-\infty$, there is a problem. The following proof deals with all cases
at once!

Write $F=(f_i)$ and $\Delta=(\delta_{ij})$ with
$f_i,\delta_{ij}\in\FKsepp$. Furthermore, write $f_i=\sum_rf_{ir}t^r$ and
$\delta_{ij}=\sum_sh_{ijs}t^s$ for $f_{ir},\delta_{ijs}\in\kp\otimes_k\Ksep$. By
multiplying the entire equation by a suitable power of $t$, we
may assume that these coefficients are zero for $r,s<0$. By assumption
we have $v_x(\delta_{ijs})\ge C$, and by definition we have $v_x(f_{ir})\neq-\infty$.

The equation $\sigma^m(F)=\Delta\cdot F$ means $\sigma^m(f_i)=\sum_{j=1}^n\delta_{ij}f_j$
for all $i$, and gives
\[\sum_{r\ge 0}\sigma^m(f_{ir})t^r=\sum_{j=1}^n\sum_{a\ge 0}\sum_{b\ge 0}\delta_{ija}f_{jb}t^{a+b}
=\sum_{r\ge 0}\left(\sum_{j=1}^n\sum_{l=0}^r\delta_{ijl}f_{j,r-l}\right)t^r.\]
From this we see that
\begin{equation}\label{eqn:anineq}
\sigma^m(f_{ir})=\sum_{j=1}^n\sum_{l=0}^r\delta_{ijl}f_{j,r-l}
\end{equation}
and must prove that $v_x(f_{ir})\ge C/(q^m-1)$. We perform induction on $r$.

If $r=0$, then for all $i$ we have $\sigma^m(f_{i0})=\sum_{j=1}^n\delta_{ij0}f_{j0}$ which
gives $q^m\cdot v_x(f_{i0})\ge\min_{j=1}^n\big(C+v_x(f_{j0})\big)$. Choosing $j$ such
that the minimum is attained we get $q^mv_x(f_{j0})\ge C+v_x(f_{j0})$ and
hence $v_x(f_{j0})\ge C/(q^m-1)$. So by the choice of $j$, for all $i$ we may deduce that
$v_x(f_{i0})\ge v_x(f_{j0})\ge C/(q^m-1)$.

For $r>0$, Equation (\ref{eqn:anineq}) gives $q^mv_x(f_{ir})\ge\inf_{j\le n,l\le r}\big(C+v_x(f_{jl})\big)$,
hence by the induction hypothesis for all $r'<r$ 
\[q^mv_x(f_{ir})\ge\min\left(\frac{q^m}{q^m-1}C,\min_{j=1}^n(C+v_x(f_{jr}))\right).\]
If $q^mC/(q^m-1)$ is smaller, we obtain $v_x(f_{ir})\ge C/(q^d-1)$ for all $i$
as in the case $r=0$. Else, choosing $j$ such that the inner minimum is attained,
we get first $v_x(f_{jr})\ge C/(q^m-1)$ and then $v_x(f_{ir})\ge C/(q^m-1)$
for all $i$, as in the case $r=0$.
\end{proof}

We now turn to the definition of our ring of periods.

\begin{dfn}\label{dfn:TamagawasBplusS}
Following \cite{Tam95}, we set
\begin{enumerate}
  \item $B^+:=\left\{f\in\FKsepp:\:\begin{array}{ll}
             v_x(f)\neq-\infty & \text{for all $x\in\Ssep$}\\
             v_x(f)\ge 0 & \text{for almost all $x\in\Ssep$}
           \end{array}\right\}$,\\``almost all'' meaning that the set of
           exceptions has finite image in $\Sigma_K$.
  \item $S:=\{s\in\OFKsepp^\times:\:\frac{\sigma(s)}{s}\in\Fp\otimes_kK\}$.
\end{enumerate}
\end{dfn}

\begin{lem}\label{lem:Bplusanditsginv}
$B^+$ is a $\GalK$-stable ring.
\end{lem}
\begin{proof}
The fact that $B^+$ is $\GalK$-stable follows directly from its definition.
That $B^+$ is a ring (closed under finite sums and products)
follows from Proposition \ref{prop:vx}: Clearly, $B^+$ contains $1$.
For $f\in B^+$ let $\Sigma_f$ denote the finite subset of those elements of $\Sigma_K$ 
over which there lies an element $x\in\Ssep$ such that $v_x(f)<0$.

Given two elements $f,g\in B^+$, for all $x\in\Ssep$ by Proposition \ref{prop:vx}(a)
 we have $v_x(f+g)\ge\min(v_x(f),v_x(g))$, which
is not equal to $-\infty$, since this is such for both $v_x(f)$ and $v_x(g)$. For
all $x$ whose image in $\Sigma_K$ does not lie in its the finite subset $\Sigma_f\cup\Sigma_g$
we even have $v_x(f+g)\ge 0$. Therefore, $f+g$ is an element of $B^+$.

A similar
proof, using Proposition \ref{prop:vx}(b), shows that $f\cdot g$ is an element of $B^+$.
All in all, $B^+$ is a ring.
\end{proof}

\begin{lem}\label{lem:Bplusanditsginv2}
$(B^+)^{\GalK}=\Fp\otimes_kK$.
\end{lem}
\begin{proof}
We note that $(B^+)^{\GalK}=B^+\cap\FKp$. So the desired equality $(B^+)^{\GalK}=\Fp\otimes_kK$
is an equality of subrings of $\FKp$.
By Lemma \ref{lem:onbfkpmodules}(a), we have $\FKp=(K')^e\llkurv t\rrkurv$
for a finite Galois extension $K'/K$ (it is Galois since $\kp\supset k$ is Galois
and $\kp\otimes_kK\isom (K')^e$). The inclusion $\FKp\subset\FKsepp$ corresponds to
a homomorphism $(K')^e\llkurv t\rrkurv\into(\Ksep)^d\llkurv t\rrkurv$ mapping
the $i$-the component of the source to $d/e$ components of the target, according to
 the $d/e$ different $K$-linear embeddings of $K'$ in $\Ksep$. It follows that the image
 of this homomorphism lies in $(K')^d\llkurv t\rrkurv$.

Given an element $f\in\FKp$, we may write it as a Laurent series $\sum_if_it^i$, with coefficients
$f_i=(f_{i1},\ldots,f_{id})\in K_r^d$. We let $V_f$ denote the $k$-vector subspace of $K_r$ generated
by the $f_{ij}$. Clearly, $\Fp\otimes_kK$ consists of those elements of $\FKp$ such that $\dim_kV_f$ is finite.

On the other hand, by definition $(B^+)^{\GalK}$ consists of those elements of $\FKp$ such that
$v_{x_r}(f)\neq-\infty$ for all $x_r\in\Sigma_{K_r}$ and $v_{x_r}(f)\ge 0$
for all but a finite number of $x_r\in\Sigma_{K_r}$.

Now, if $f\in\FKp$ is an element of $\Fp\otimes_kK$, then $\dim_kV_f$ is finite,
so the subset of $\Sigma_{K_r}$ consisting of the poles of the (coefficients of the) elements of $V_f$ is finite, so
$f$ is an element of $B^+$ by our above characterisation.

On the other hand, if $f\in\FKp$ is an element of $B^+$, then we may choose
a finite subset $\Sigma_0\subset\Sigma_{K_r}$ such that $v_{x_r}(f)\ge 0$ for all
$x_r\not\in\Sigma_0$. For $x_r\in\Sigma_0$, we set $n(x_r):=-v_{x_r}(f)$, which
is finite by assumption. Let $X_r$ denote the proper normal variety
over $k$ corresponding to $K_r$. Since $X_r$ proper, the space of global sections
of
\[\mathcal{O}_{X_r}\left(\sum_{x_r\in\Sigma_0} n(x_r)x_r\right)\]
is finite-dimensional. Since it contains $V_f$, this implies that $f\in\Fp\otimes_kK$
by our above characterisation.
\end{proof}

\begin{lem}\label{lem:bplus}
$S$ is a $\GalK$-stable multiplicative subset of $B^+$.
\end{lem}
\begin{proof}
The fact that $S$ is a $\GalK$-stable multiplicative subset of $\FKsepp$ follows directly
from its definition.

Let us show that $S$ is contained in $B^+$.
For $s\in S$ choose $f\in\Fp\otimes_kK$ such that $\sigma(s)=f\cdot s$, such
an $f$ exists by definition of $S$. By
Lemma \ref{lem:Bplusanditsginv2} and Proposition \ref{cor:fixpt}, $v_x(s)\neq-\infty$ for all $x\in\Ssep$,
and there exists a finite subset $\Sigma_0$ of $\Sigma_K$
such that $v_x(f)\ge 0$ for all $x\in\Ssep$ not lying over $\Sigma_0$.

For all $x\in\Ssep$, Proposition \ref{cor:fixpt} shows that $v_x(s)\ge v_x(f)/(q-1)$.
So $s$ has the required properties that $v_x(s)\neq-\infty$ for all $x\in\Ssep$
and $v_x(s)\ge 0$ for all $x\in\Ssep$ not lying over $\Sigma_0$, since
this is the case for $f$.
\end{proof}

\begin{dfn}\label{dfn:TamagawasB}
Following \cite{Tam95}, we let $B\subset\FKsepp$ be the ring obtained by inverting $S\subset B^+$,
and set $\bB=(B,\sigma)$, where $\sigma$ is the given ring endomorphism of $\bFKsepp$.
\end{dfn}

\begin{lem}\label{lem:partofA}
$\bB$ is a bold ring with ring of scalars $\Fp$.
\end{lem}
\begin{proof}
$B$ is clearly $\sigma$-stable since $B^+$ and $S$ are. Furthermore,
since $\Fp\subset B$ and $B^\sigma\subset\FKsepp^\sigma=\Fp$, we
have $B^\sigma=\Fp$.
\end{proof}

We say that an element $f\in\FKsepp$ has \emph{order} $n\in\bbZ$
if, writing $f$ as $\sum f_it^i\in(\kp\otimes_k\Ksep)\llkurv t\rrkurv$ we have
$n=\inf\{i:\:f_i\neq 0\}$. We say that an element $f\in\FKsepp$ of order $n$
has \emph{invertible leading coefficient} if $f_n$ is invertible in $\kp\otimes_k\Ksep$.
If $f$ has order $0$, then we will denote by $f(0)$ the leading coefficient of $f$.
Note that the invertible elements of $\OFKsepp$ are precisely the elements of
$\FKsepp$ of order $0$ with invertible leading coefficient.

\begin{rem}\label{rem:invertfksepp}
Let us set $t_i:=e_i\cdot t\in\FKsepp$, where $e_i$
is the standard basis vector of the $i$-th copy of $\Ksep$ in the product $(\Ksep)^d$.
Clearly, an element $f\in\FKsepp$ is invertible if and only if
we can write
\[f=\left(\prod_{i=0}^{d-1}t_i^{n_i}\right)\cdot\widetilde{f},\]
for certain $n_i\in\bbZ$, where $\widetilde{f}$ is an element of $\OFKsepp^\times$.
\end{rem}

\begin{lem}\label{lem:arschrey}
Every element $f\in\OFKsepp^\times$ may be written as $f=\frac{\sigma(s)}{s}$
for some other element $s\in\OFKsepp^\times$.
\end{lem}
\begin{proof}
We write $f=\sum_{i\ge 0} f_it^i$ and use the ``ansatz'' $s=\sum_{j\ge0}s_jt^j$.
This gives
\[\sum_r\sigma(s_r)t^r=\sigma(s)=sf=\sum_{i,j}f_is_jt^{i+j}=\sum_r\left(\sum_{i=0}^rf_is_{r-i}\right)t^r.\]
We proceed by induction. For $r=0$, we must solve $\sigma(s_0)=f_0s_0$. We write
$f_0=(f_{0,0},\ldots,f_{0,d-1})$ and $s_0=(s_{0,0},\ldots,s_{0,d-1})$ for $f_{0,i},s_{0,i}\in\Ksep$.
Note that by assumption all $f_{0,i}\neq 0$.
Since
\[\sigma(s_0)=(s_{0,d-1}^q,s_{0,0}^q,s_{0,1}^1,\ldots,s_{0,d-1}^q)\]
our equation $\sigma(s_0)=f_0s_0$ is equivalent to the system of equations
\[s_{0,i}^q=f_{0,i+1}s_{0,i+1},\qquad i\in \bbZ/d\bbZ.\]
This means, for instance, that $s_{0,0}=s_{0,d-1}^q/f_{0,0}$ and $s_{0,d-1}=s_{0,d-2}^q/f_{0,d-1}$,
which gives
\[s_{0,0}^q=\frac{s_{0,d-1}^q}{f_{0,0}}=\frac{\left(s_{0,d-2}^q/f_{0,d-1}\right)^q}{f_{0,0}}.\]
Iterating this substitution, we obtain the equation
\[s_{0,0}^{q^d}-\left(f_{0,1}^{q^{d-1}}\cdot f_{0,2}^{q^{d-2}}\cdots f_{0,d-1}^q\cdot f_{0,0}\right)s_{0,0}=0.\]
Since all the $f_{0,i}\neq 0$, the constant $\phi:=f_{0,1}^{q^{d-1}}\cdot f_{0,2}^{q^{d-2}}\cdots f_{0,d-1}^q\cdot f_{0,0}$
is non-zero, so this is a separable equation for $s_{0,0}$ and hence has a non-trivial solution in $\Ksep$.
The $s_{0,i}$ for $i\neq 0$ are then determined by the assignments $s_{0,i}:=s_{0,i-1}^q/f_{0,i}$,
they are non-trivial since $s_{0,0}$ and the $f_{0,i}$ are.

Let us consider the case $r>0$, and write $s_r=(s_{r,0},\ldots,s_{r,d-1})$ and $f_r=(f_{r,0},\ldots,f_{r,d-1})$.
In this case, the equation $\sigma(s_r)=\sum_{i=0}^rf_is_{r-i}$
that we must solve is equivalent to the system of equations
\[s_{r,i+1}^q=\sum_{j=0}^rf_{i,0}s_{r-i,0}=:f_{0,i}s_{r,i}+C_{r,i},\]
where the $C_{r,i}\in\Ksep$ are constants dependant only on $f$ and the $s_{r'}$ for $r'<r$.

We may use the same type of replacement as before, and obtain an equation
\[s_{r,0}^{q^d}-\phi\cdot s_{r,0}=C_r\]
with $C_r\in\Ksep$ a constant determined by the $C_{r,i}$. Again, this is a separable equation for $s_{r,0}$,
so there exists a solution in $\Ksep$. The $s_{r,i}$
for $i\neq 0$ are then determined by the equations $s_{r,i}=(s_{r,i+1}^q-C_{r,i})/f_{0,i}$.

Finally, since we may choose the $s_{0,i}$ to be non-zero, our solution $s$ is in
fact invertible in $\OFKsepp$.
\end{proof}

\begin{prop}\label{prop:partofB}
$B$ is a $\GalK$-stable ring, and $B^{\GalK}\supset\FpK$.
\end{prop}
\begin{proof}
$B$ is clearly $\GalK$-stable, since $B^+$ and $S$ both are. We have $B^{\GalK}=B\cap\FKp$.

Let us show that $\FpK\subset B$. Consider
$g/f\in\FpK$ with $f,g\in\Fp\otimes_kK$. By Remark \ref{rem:invertfksepp},
we may assume that $f$ is in $\OFKsepp^\times$.
By Lemma \ref{lem:arschrey} there exists an element $s\in S$
with $f=\sigma(s)/s$. It follows that $g/f=gs/\sigma(s)\in B$,
since $gs\in B^+$ by Lemma \ref{lem:Bplusanditsginv} and $\sigma(s)\in S$.
\end{proof}

We turn to the inclusion $B^{\GalK}\subset\FpK$, which is
more difficult. Consider $b=b^+/s\in B^{\GalK}$, with
$b^+\in B^+$ and $s\in S\subset\OFKsepp^\times$. We set $f:=\sigma^d(s)/s$,
which is an element of $\Fp\otimes_kK$,
and for $N\ge 0$ -- following \cite{Tam04} -- we set
\[a_N:=b\cdot f\left(t^{q^d}\right)\cdot f\left(t^{q^{2d}}\right)\cdots f\left(t^{q^{Nd}}\right)\in\FKp.\]

\begin{rem}\label{rem:stratforsth}
Our goal is to show that for $N$ large enough
the element $a_N$ lies in $B^+$. By Lemma \ref{lem:bplus}
this will imply that $a_N\in\Fp\otimes_kK$, and in particular
that $b\in B$.
\end{rem}

\begin{lem}\label{lem:BGlemA}
There exists a finite set $\Sigma_0\subset\Sigma_K$ such that
for all $N\ge 0$ and all $x\in\Ssep$ not lying above $\Sigma_N$ we have $v_x(a_N)\ge 0$.
\end{lem}
\begin{proof}
The idea is to use that $b^+$, $s$ and $f$ all lie in $B^+$, and then apply
Proposition \ref{prop:vx}(b). In order to handle $1/s$, which is not necessarily
an element of $B^+$, we need some modifications. Let $s(0)$ denote the leading
coefficient of $s$, and set $\tils:=s/s(0)$. Clearly, $\tils$
is an element of $S$ with leading coefficient $1$. Setting $\tilf:=\sigma^d(\tils)/\tils$,
we have $\tilf\in\Fp\otimes_kK$ and $f=\mu\cdot\tilf$ with $\mu:=\sigma^d(s(0))/s(0)$ an invertible
element of $\kp\otimes_kK$. Now by definition and Proposition \ref{prop:vx}(b), we have
\begin{eqnarray*}
v_x(a_N)&=&v_x\left(\frac{b^+}{s}\cdot f\left(t^{q^d}\right)\cdots f\left(t^{q^{Nd}}\right)\right)\\
        &=&v_x\left(\frac{\mu^N}{s(0)}\cdot b^+ \cdot \frac{1}{\tils} \cdot \tilf\left(t^{q^d}\right)\cdots \tilf\left(t^{q^{Nd}}\right)\right)\\
        &\ge& N\cdot v_x(\mu) + v_x\left(\frac{1}{s(0)}\right) + v_x(b^+) + v_x\left(\frac{1}{\tils}\right) + N\cdot v_x(\tilf)
\end{eqnarray*}
Since $E:=\{\,\mu,1/s(0),b^+,\tilf\,\}$ is a finite subset of $B^+$, the
set $\Sigma_0'$ of those $x\in\Ssep$ for which there exists an $e\in E$ such that $v_x(e)<0$
has finite image in $\Sigma_K$. Call this image $\Sigma_0$, and consider any $x\in\Sigma_0$.
Proposition \ref{cor:fixpt} implies that $v_x(\tils)\ge v_x(\tilf)/(q^d-1)\ge 0$.
Since $\tils$ has leading coefficient $1$, we may calculate $1/\tils$ via the geometric
series, and obtain $v_x(1/\tils)\ge 0$, using Proposition \ref{prop:vx}.
Therefore, $v_x(a_N)$ is bounded below by a finite sum of non-negative numbers, so $v_x(a_N)\ge 0$
for all $x$ not lying above $\Sigma_0$.
\end{proof}

\begin{lem}[following \cite{Tam94b}]\label{slem:eps}
Let $s\in\OFKsepp^\times$, $x\in\Ssep$ and $N\ge 0$ fulfill
\begin{enumerate}
  \item $v_x(s)\ge 0$, and
  \item $v_x(s(0))<q^N$.
\end{enumerate}
Then, for every $a\in\FKp$ we have an inequality
\[v_x\left(\sigma^N(a)\right)\ge \left\lfloor\frac{v_x\left(s\cdot\sigma^N(a)\right)}{q^N}\right\rfloor\cdot q^N,\]
where for $x\in\bbR$ the term $\lfloor x\rfloor$ denotes the largest integer smaller than $x$.
\end{lem}
\begin{proof}
We write $s=\sum_{i\ge 0}s_it^i$ and $b:=\sigma^N(a)=\sum_ib_it^i$ with coefficients
$s_i\in\kp\otimes_k\Ksep$ and $b_i\in \kp\otimes_kK$. We may assume that $b_i=0$ for $i<0$.
By assumption, $v_x(s_i)\ge 0$ for all $i$, and $v_x(s_0)<q^N$. Note
that since $s_0$ is invertible, the inequality $v_x(s_0\cdot b_i)\ge v_x(s_0)+v_x(b_i)$
is in fact an equality!

We set $C:=\lfloor v_x(sb)/q^N\rfloor\cdot q^N$, must prove that $v_x(b_i)\ge C$ for all $i$, and do this by induction on $i$.

For $i=0$, we consider the inequality $v_x(s_0)+v_x(b_0)= v_x(s_0b_0)\ge C$.
It implies that, $v_x(b_0)\ge C-v_x(s_0)>C-q^N$. However, by assumption the value of $v_x(b_0)$
lies in $q^N\cdot\bbZ\cup\{\infty\}$, and there exists no integral
multiple of $q^N$ strictly greater than $C-q^N$ and less than $C$. Therefore,
we have $v_x(b_0)\ge C$.

For $i>0$, we have $s_0b_i=(s b)_i-\sum_{j=1}^is_{j}b_{i-j}$. By induction,
we deduce that
\begin{eqnarray*}
v_x(s_0b_i)&=&v_x\left((s b)_i-\sum_{j=1}^is_{j}b_{i-j}\right)\\
&\ge&\min\left(v_x\,\big((sb)_i\big),\:\min_{1\le j\le i}\Big(v_x(s_j)+v_x(b_{i-j}\Big)\right)\\
&\ge&\min(C,\min(0+C))\ge C
\end{eqnarray*}
So $v_x(b_i)\ge C-v_x(s_0)$, which implies that $v_x(b_i)\ge C$ as in the case $i=0$
since $v_x(b_i)$ is an integral multiple of $q^N$ and $0\le v_x(s_0)<q^N$.
\end{proof}

\begin{lem}\label{lem:BGlemB}
There exists an $N_0\ge 1$ such that for all $N\ge N_0$ and all $x\in\Ssep$
we have $v_x(a_N)\neq-\infty$.
\end{lem}
\begin{proof}
By Lemma \ref{lem:BGlemB}, there exists a finite set $\Sigma_0\subset\Sigma_K$ 
such that $v_x(a_N)\ge 0>-\infty$ for all $x$ not lying above $\Sigma_0$.
Hence it suffices to prove that, for one given $x_K\in\Sigma_K$, there exists
an integer $N_0\ge 1$ such that for all $N\ge N_0$ and all $x$ lying above
$x_K$ we have $v_x(a_N)\neq-\infty$. We fix such an $x_K\in\Sigma_0$.

Let $\pi$ denote a local parameter of $K$ at $x_K$. For all $x$ over $x_K$, we
have $v_x(s)\ge v_x(f)/(q^d-1)>-\infty$ by Proposition \ref{cor:fixpt},
so that $s=\pi^{-n}\widetilde{s}$ for some $n\ge0$ and
$\widetilde{s}\in S$ satisfying $v_x(s)\ge 0$. As a first substep, we wish to
show that it is sufficient to deal with the case $s=\widetilde{s}$.
This will make our calculations easier!

If $n>0$, then
\[\widetilde{f}:=\frac{\sigma^d(\widetilde{s})}{\widetilde{s}}=\frac{\sigma^d(\pi^{n})}{\pi^{n}}\cdot\frac{\sigma(s)}{s}=\pi^{n(q^d-1)}f\in\Fp\otimes_kK,\]
and by setting $\widetilde{b^+}:=\pi^nb^+\in B^+$, we obtain $b=\widetilde{b^+}/\widetilde{s}$,
so that
\begin{eqnarray*}
\widetilde{a_N}&:=&b\cdot \widetilde{f}(t^{q^d})\cdots\widetilde{f}\big(t^{q^{Nd}}\big)\\
&=&b\cdot\pi^{n(q^d-1)}f(t^{q^d})\cdots\pi^{n(q^d-1)}f\big(t^{q^{Nd}}\big)\\
&=&\pi^{Nn(q^d-1)}a_N.
\end{eqnarray*}
In particular, $v_x(a_N)\neq-\infty$ if and only if $v_x(\widetilde{a_N})\neq-\infty$,
and we may assume in the following without loss of generality that the $s\in\OFKsepp^\times$ we
are given fulfills $v_x(s)\ge 0$.

We remark that for all $g\in\FKsepp$ and $i\ge 0$ we have
the formula
\begin{equation}\label{eqn:sigmaandpowers}
\sigma^{id}(g(t^{q^{id}}))=g^{q^{id}},
\end{equation}
in particular for our given $f\in\Fp\otimes_kK$.

Secondly, note that from $b^+=bs$ and $\sigma^d(s)=sf$ we obtain
$\sigma^d(b^+)=\sigma^d(b)\sigma^d(s)=\sigma^d(b)sf$,
and by induction for $N\ge 1$
\begin{equation}\label{eqn:iteratesigma}
\sigma^{Nd}(b^+)=\sigma^{Nd}(b)s\cdot\left(f\cdot\sigma^d(f)\cdots\sigma^{(N-1)d}(f)\right).
\end{equation}
Hence,
\begin{eqnarray*}
\sigma^{Nd}(a_N)s&=&\sigma^{Nd}\left(b\cdot f(t^{q^d})\cdots f(t^{q^{Nd}})\right)\cdot s\\
&=&\sigma^N(b)s\cdot\sigma^{Nd}\left(f(t^{q^d})\cdots f(t^{q^{Nd}})\right)\\
&=&\sigma^N(b^+)\cdot\frac{\sigma^{Nd}\left(f(t^{q^d})\cdots f(t^{q^{Nd}})\right)}{\sigma^{(N-1)d}(f)\cdots f}\quad\text{by Equation (\ref{eqn:iteratesigma})}\\
&=&\sigma^N(b^+)\cdot\prod_{i=1}^N\sigma^{(N-i)d}\left(\frac{\sigma^{id}(f(t^{q^{id}}))}{f}\right)\\
&=&\sigma^N(b^+)\cdot\prod_{i=1}^N\sigma^{(N-i)d}\left(f^{q^{id}-1}\right)\quad\text{by Equation (\ref{eqn:sigmaandpowers})}\\
&=:&\sigma^N(b^+)\cdot\phi,
\end{eqnarray*}
with $\phi\in\Fp\otimes_kK$, so it follows that $v_x(\sigma^N(a_N)s)\ge q^Nv_x(b^+)+v_x(\phi)\neq-\infty$.

Now if $N$ is large enough, namely, $q^N>v_x(s(0))$, then Lemma \ref{slem:eps} shows that
$q^Nv_x(a_N)=v_x(\sigma^N(a_N))\neq-\infty$, so $v_x(a_N)\neq-\infty$ as required.
\end{proof}

\begin{prop}\label{prop:part2ofB}
The ring $B$ fulfills $B^{\GalK}=\FpK$.
\end{prop}
\begin{proof}
By Proposition \ref{prop:partofB} it suffices to show
that $B^\GalK\subset\FpK$. For $b\in B^{\GalK}$ and $N\ge 0$, define $a_N$ as
before Remark \ref{rem:stratforsth}. Lemmas \ref{lem:BGlemA} and \ref{lem:BGlemB}
show that for $N$ large enough, $a_N$ is an element of $B^+$. By construction,
it is an $\GalK$-invariant, so Lemma \ref{lem:Bplusanditsginv} shows that
$a_N\in\Fp\otimes_kK$. By definition, this shows that
\[b=\frac{a_N}{f\left(t^{q^d}\right)\cdot f\left(t^{q^{2d}}\right)\cdots f\left(t^{q^{Nd}}\right)}\]
is an element of $\FpK$, since both $a_N$ and the denominator lie in
$\Fp\otimes_kK\subset\FpK$.
\end{proof}

So far, we have shown that $\bB$ is a well-defined $\GalK$-stable bold
ring with scalar ring $\Fp$ and $B^{\GalK}=\FKp$. It remains to prove
that $B$ has property (c) of Claim \ref{claim:Bexists}.

\begin{lem}\label{lem:claimc}
Let $\bM$ be a $\frkp$-restricted $\bFpK$-module. Then $\Vp(\bM)\subset \bB\otimes_{\bFpK}\bM$.
\end{lem}
\begin{proof}
We may assume, by choosing a basis,
that $\bM=(\FpK^{\oplus n},\tau)$ with
$\tau(m)=\Delta\sigma(m)$ for some matrix $\Delta\in\GL_n(\FpK)$ and all $m$.

Since $\Vp(\bM)=(\bFKsepp\otimes\bM)^{\tau}$, we have to prove that for all
$m\in\FKsepp^{\oplus n}$ the equation $\Delta\cdot\sigma(m)=m$ implies
that all entries of $m$ lie in $B$.

Let us denote the inverse of $\Delta$ by $\Delta^{-1}=(g_{ij}/f_{ij})_{i,j}$,
with $g_{ij}\in\Fp\otimes_kK$ and $f_{ij}\in(\Fp\otimes_kK)\cap\OFKsepp^\times$.
Setting $f:=\prod_{i,j}f_{ij}$, we see that $\Delta^{-1}=\frac{1}{f}\Delta'$
for some matrix $\Delta'$ with entries in $\Fp\otimes_kK\subset B^+$.

By Lemma \ref{lem:arschrey}, we may write $f=\sigma(s)/s$ for some
$s\in S$. For any element $m\in M$ write $m':=sm$. Now the equation $\tau(m)=m$
is equivalent to the equation $\sigma(m')=\Delta'\cdot m'$.
By Proposition \ref{cor:fixpt}, this implies that $m'$ has entries in $B^+$,
so in particular $m=m'/s$ has entries in $B$, as claimed.
\end{proof}

\begin{thm}\label{thm:Bexists}
The ring $B$ fulfills Claim \ref{claim:Bexists}.
\end{thm}
\begin{proof}
By construction, $B$ is a subring of $\FKsepp$.
By Lemma \ref{lem:partofA}, it fulfills Claim \ref{claim:Bexists}(a).
By Propositions \ref{prop:partofB} and \ref{prop:part2ofB}, it
fulfills Claim \ref{claim:Bexists}(b). By Lemma \ref{lem:claimc}, it
also fulfills Claim \ref{claim:Bexists}(c).
\end{proof}

\section{Algebraic Monodromy Groups}

We recall the setup of Tannakian duality.

\begin{dfn}
\begin{enumerate}
  \item Let $F$ be a field. A \emph{pre-Tannakian category over $F$} is
  an $F$-linear rigid tensor category $\mcT$ such that all objects are of
  finite length, and for which the natural homomorphism $F\to\End_{\mcT}(\bbu)$
  is an isomorphism.
  \item Let $\mcT$ be a pre-Tannakian category, and consider an object
  $X$ of $\mcT$. Then $\llkurv X\rrkurv_\otimes$ denotes the smallest
  full abelian subcategory of $\mcT$ closed under tensor products
  and subquotients in $\mcT$.
  \item Let $\mcT$ be a pre-Tannakian category over $F$. Let $F'/F$ be a field
  extension. A \emph{fibre functor} on $\mcT$ is a faithful $F$-linear exact
  tensor functor $\omega:\:\mcT\to\Vect_{F'}$, where $\Vect_{F'}$ denotes the
  category of finite-dimensional $F'$-vector spaces. If $F'=F$, the fibre functor
  is called \emph{neutral}.
  \item A \emph{Tannakian category over $F$} is a pre-Tannakian category for which
  there exists a fibre functor over some field extension $F'/F$.
  \item Let $\mcT$ be a Tannakian category over $F$, consider a fibre functor
  $\omega$ of $\mcT$ over $F'/F$, and fix an object $X$ of $\mcT$. The
  \emph{algebraic monodromy group} of $X$ with respect to $\omega$
  is the functor
  \[G_{\omega}(X):\:\left(\!\left(\text{$F'$-algebras}\right)\!\right)\To\left(\!\left(\text{groups}\right)\!\right),\]
  mapping an $F'$-algebra $R'$ to the group of tensor automorphisms
  of the functor $R'\otimes_{F'}\omega(-)$ from $\llkurv X\rrkurv_\otimes$ to $R'$-modules.
\end{enumerate}
\end{dfn}

\begin{prop}
Let $\mcT$ be a Tannakian category over $F$, consider a fibre functor
$\omega$ of $\mcT$ over $F'/F$, and fix an object $X$ of $\mcT$. Then
the algebraic monodromy group of $X$ with respect to $\mcT$ is
representable by an affine group scheme over $F'$.
\end{prop}
\begin{proof}
\cite[Theorem $3.1.7$(a)]{Sta08}. This seems to be well-known (to the experts).
\end{proof}

Let $F,\Fq,A,K,\iota$ be as in Section $2$, and choose a maximal
ideal $\frkp\neq\ker\iota$. In Section $2$, we have constructed
the category $\AIsomotK$ of $A$-isomotives over $K$. Using either
the results of Section $2$, or the embedding $I$ of Proposition
\ref{prop:embedaisomotk}, we see that it is a pre-Tannakian category.
The category $\RepFpGalK$ is a Tannakian category, since it fulfills the
properties required by a pre-Tannakian category, and the
forgetful functor $U:\:\RepFpGalK\to\Vect_{\Fp}$ is a fibre functor.

In Section $2$, we also constructed the functor
\[\Vp=R_{\frkp}\circ\big(\FKp\otimes_{\FpK}(-)\big)\circ\big(\FpK\otimes_{\FK}(-)\big)\circ I:\:\AIsomotK\to\RepFpGalK,\]
associating to an $A$-isomotive its rational Tate module. It
is faithful, $F$-linear and exact as a composition of
such functors. Therefore, $\AIsomotK$ is Tannakian, with fibre functor $U\circ\Vp$.

Given an $A$-isomotive $\bX$, we set $G_{\frkp}(\bX):=G_{U\circ\Vp}(X)$,
the \emph{algebraic monodromy group} of $X$ at $\frkp$.
On the other hand, we may consider $\Gamma_{\frkp}(\bX)$, the image
of $\GalK:=\Gal(\Ksep/K)$ in $\Aut_{\Fp}\big(\Vp(\bX)\big)$. This
might be called the \emph{$\frkp$-adic monodromy group} of
$\bX$, or rather $\Vp(\bX)$.

\begin{prop}\label{prop:monodromygroupsofabstractgroups}
Let $F'$ be a field, $V$ be a finite-dimensional $F'$-vector space
and consider a subgroup $\Gamma\subset\GL(V)(F')$ with associated
algebraic group $G:=\overline{\Gamma}^{\mathrm{Zar}}\subset\GL(V)$.
The natural homomorphism $G\to G_U(V)$, with target
the the algebraic monodromy
group of $V$ as a representation of $G$ with respect to
the forgetful fibre functor $U$, is an isomorphism.
\end{prop}
\begin{proof}
\cite[Proposition 3.3.3(b)]{Sta08}. This seems
to be well-known (to the experts).
\end{proof}

It follows that $\Gamma_{\frkp}(\bX)$ is Zariski-dense subgroup
of the group of $\Fp$-rational
points of the algebraic monodromy group of $\Vp(\bX)$ with respect
to the forgetful fibre functor $U$ of $\RepFpGalK$. In order
to prove Theorem \ref{thm:secondmainthm}(a), we must compare
$G_{U\circ\Vp}(\bX)$ and $G_U(\Vp\bX)$. It is here
that we invoke one of the main results of my article \cite{Sta08}.

\begin{thm}\label{thm:sta08y}
Let $F'/F$ be a separable field extension, $\mcT$ a Tannakian
category over $F$, $\mcT'$ a Tannakian category over $F'$
and $\omega:\:\mcT'\to\Vect_{F'}$ a neutral fibre functor.
Let $V:\:\mcT\to\mcT'$ be an $F$-linear exact functor which
is $F'/F$-fully faithful, and semisimple on objects.

For every object $X$ of $\mcT$ the natural homomorphism $G_{\omega}\big(V(X)\big)\to G_{\omega\circ V}(X)$
is an isomorphism of algebraic groups.
\end{thm}
\begin{proof}
\cite[Proposition $3.1.8$]{Sta08}.
\end{proof}

\begin{proof}[Proof of Theorem \ref{thm:secondmainthm}(a)]
Theorem \ref{thm:mainthm} and Proposition \ref{prop:tateconj}
show that $\Vp$ has the properties
required in Theorem \ref{thm:sta08y}. Together with
Proposition \ref{prop:monodromygroupsofabstractgroups}, we
see that the image of $\Gamma_{\frkp}(\bX)\to G_{\frkp}(\bX)(\Fp)$
is indeed Zariski-dense in $G_{\frkp}(\bX)$ for every
$A$-isomotive $\bX$.
\end{proof}

\begin{dfn}
A semisimple $F$-algebra $E$ is \emph{separable} if the
center of each simple factor of $E$ is a separable field
extension of $F$.
\end{dfn}

\begin{prop}\label{thm:redcrit}
Let $F'$ be a field, $V$ a finite-dimensional $F'$-vector space, and consider a closed algebraic subgroup $G\subset\GL(V)$.
If $V$ is semisimple as a representation of $G$, and $\End_G(V)$ is a separable $F'$-algebra, then the identity
component $G^\circ$ is a reductive group.
\end{prop}
\begin{proof}
\cite[Proposition $3.2.1$]{Sta08}.
This seems to be well-known (to the experts).
\end{proof}

\begin{proof}[Proof of Theorem \ref{thm:secondmainthm}(b)]
Let $\bX$ be a semisimple $A$-isomotive with separable
endomorphism algebra. By Theorem \ref{thm:secondmainthm}(a)
the algebraic monodromy group $G:=G_{\frkp}(\bX)$ acts
faithfully on $\Vp(\bX)$, the rational Tate module of $\bX$.
Since $\Vp$ is fully faithfull by Proposition \ref{prop:tateconj},
$\End_G(\Vp\bX)\isom\Fp\otimes_F\End(\bX)$,
so this is a semisimple separable $\Fp$-algebra by
\cite[no. 7, \S 5, Proposition 6, Corollaire]{BourbA}.
Therefore, Proposition \ref{thm:redcrit} implies that $G^\circ$
is indeed a reductive group.
\end{proof}

\bibliographystyle{alpha}

\end{document}

%% file: preamble.tex
\newcommand{\bbF}{\mathbb{F}}

\newcommand{\bbQ}{\mathbb{Q}}
\newcommand{\bbR}{\mathbb{R}}

\newcommand{\bbZ}{\mathbb{Z}}
\newcommand{\bbu}{\boldsymbol{1}}

\newcommand{\mcA}{\mathcal{A}}

\newcommand{\mcB}{\mathcal{B}}

\newcommand{\mcC}{\mathcal{C}}

\newcommand{\mcD}{\mathcal{D}}

\newcommand{\mcH}{\mathcal{H}}

\newcommand{\mcT}{\mathcal{T}}

\newcommand{\Fq}{\mathbb{F}_q}

\newcommand{\bB}{\boldsymbol{B}}

\newcommand{\bOB}{\boldsymbol{B'}}

\newcommand{\bM}{\boldsymbol{M}}

\newcommand{\bN}{\boldsymbol{N}}
\newcommand{\bV}{\boldsymbol{V}}
\newcommand{\bW}{\boldsymbol{W}}

\newcommand{\bX}{\boldsymbol{X}}

\newcommand{\frka}{\mathfrak{a}}

\newcommand{\frkp}{\mathfrak{p}}

\newcommand{\frkP}{\mathfrak{P}}

\DeclareMathOperator{\id}{id}
\DeclareMathOperator{\coker}{coker}
\DeclareMathOperator{\End}{End}

\DeclareMathOperator{\Mod}{Mod}

\DeclareMathOperator{\Hom}{Hom}

\DeclareMathOperator{\im}{im}

\DeclareMathOperator{\Tp}{T_{\frkp}}
\DeclareMathOperator{\Vp}{V_{\frkp}}
\DeclareMathOperator{\rk}{rk}
\DeclareMathOperator{\Gal}{Gal}

\DeclareMathOperator{\Mat}{Mat}
\DeclareMathOperator{\Ann}{Ann}
\DeclareMathOperator{\GL}{GL}

\DeclareMathOperator{\Aut}{Aut}

\DeclareMathOperator{\Frac}{Frac}
\DeclareMathOperator{\den}{den}
\DeclareMathOperator{\lcm}{lcm}

\DeclareMathOperator{\pr}{pr}

\DeclareMathOperator{\Rp}{R_\frkp}
\DeclareMathOperator{\Cp}{C_\frkp}
\DeclareMathOperator{\OCp}{C_\frkp^{\prime}}
\DeclareMathOperator{\Dp}{D_\frkp}

\newcommand{\isom}{\cong}
\newcommand{\llkurv}{(\!(}
\newcommand{\rrkurv}{)\!)}
\newcommand{\lleck}{[\![}
\newcommand{\rreck}{]\!]}
\newcommand{\Ssep}{\Sigma^{\mathrm{sep}}}

\newcommand{\To}{\longrightarrow}

\newcommand{\kp}{\bbF_{\frkp}}
\newcommand{\Fp}{F_{\frkp}}
\newcommand{\Ap}{A_{\frkp}}
\newcommand{\ktt}{\Fq\llkurv t\rrkurv}
\newcommand{\Ksep}{K^{\mathrm{sep}}}

\newcommand{\FK}{F_K}

\newcommand{\AK}{A_K}

\newcommand{\FpK}{F_{\frkp,K}}

\newcommand{\FKp}{F_{K,\frkp}}

\newcommand{\OFKsepp}{A_{\Ksep,\frkp}}
\newcommand{\bAKsepp}{\boldsymbol{\AKsepp}}
\newcommand{\FKsepp}{F_{\Ksep,\frkp}}

\newcommand{\OB}{B'}
\newcommand{\AKsepp}{A_{\Ksep,\frkp}}

\newcommand{\bAK}{\boldsymbol{A_K}}

\newcommand{\bFK}{\boldsymbol{F_K}}

\newcommand{\bFFK}{\bFK\!\!\!^{\prime\phantom{.}}}
\newcommand{\bFFFK}{\bFK\!\!\!^{\prime\prime\phantom{.}}}
\newcommand{\FFK}{F_K'}

\newcommand{\bFpK}{\boldsymbol{\FpK}}
\newcommand{\bFKp}{\boldsymbol{\FKp}}

\newcommand{\bFKsepp}{\boldsymbol{F}_{\boldsymbol{\Ksep},\frkp}}

\newcommand{\bR}{\boldsymbol{R}}
\newcommand{\bS}{\boldsymbol{S}}
\newcommand{\bT}{\boldsymbol{T}}

\newcommand{\tils}{\widetilde{s}}
\newcommand{\tilf}{\widetilde{f}}

\newcommand{\GalK}{{\Gamma\!_K}}

\newcommand{\Gp}{G_{\frkp}}

\newcommand{\taulin}{\tau_{\mathrm{lin}}}

\newcommand{\bMcal}{\boldsymbol{\mathcal{M}}}
\newcommand{\bNcal}{\boldsymbol{\mathcal{N}}}

\DeclareMathOperator{\Qp}{C_{\frkp}}

\newcommand{\TFp}{\Cp}
\newcommand{\OQp}{C_{\frkp}^{\prime}}

\newcommand{\AlpK}{A_{(\frkp),K}}
\newcommand{\bAlpK}{\boldsymbol{\AlpK}}

\DeclareMathOperator{\Rep}{Rep}
\DeclareMathOperator{\Vect}{Vec}

\newcommand{\RepApGalK}{\Rep_{\Ap}(\GalK)}
\newcommand{\RepFpGalK}{\Rep_{\Fp}(\GalK)}
\DeclareMathOperator{\Mot}{Mot}
\DeclareMathOperator{\Isomot}{Isomot}
\newcommand{\AMotKeff}{\text{$A$-$\Mot_K^\mathrm{eff}$}}
\newcommand{\AMotK}{\text{$A$-$\Mot_K$}}

\newcommand{\AIsomotK}{\text{$A$-$\Isomot_K$}}

\newcommand{\bRMod}{\text{$\bR$-$\Mod$}}
\newcommand{\bSMod}{\text{$\bS$-$\Mod$}}

\newcommand{\AKMod}{\text{$\bAK$-$\Mod$}}

\newcommand{\FKMod}{\text{$\bFK$-$\Mod$}}
\newcommand{\FKpMod}{\text{$\bFKp$-$\Mod$}}

\newcommand{\FKModpet}{\text{$\bFK$-$\Mod^{\frkp\text{-res}}$}}
\newcommand{\FpKModres}{\text{$\bFpK$-$\Mod^{\frkp\text{-}\mathrm{res}}$}}
\newcommand{\FpKModet}{\FpKModres}
\newcommand{\FKpModres}{\text{$\bFKp$-$\Mod^{\frkp\text{-}\mathrm{res}}$}}
\newcommand{\FKpModet}{\FKpModres}

\newcommand{\AKp}{A_{K,\frkp}}
\newcommand{\ApK}{A_{\frkp,K}}
\newcommand{\bAKp}{\boldsymbol{\AKp}}
\newcommand{\bApK}{\boldsymbol{\ApK}}

\numberwithin{equation}{section}

\theoremstyle{plain} 
\newtheorem{cor}[equation]{Corollary}
\newtheorem{lem}[equation]{Lemma}

\newtheorem{prop}[equation]{Proposition}
\newtheorem{thm}[equation]{Theorem} 
\newtheorem{claim}[equation]{Claim}

\theoremstyle{definition} 
\newtheorem{dfn}[equation]{Definition}
\newtheorem{ex}[equation]{Example}

\theoremstyle{remark} 
\newtheorem{rem}[equation]{Remark}

\newtheorem{constr}[equation]{Construction}

\mathchardef\ordinarycolon\mathcode`\:
\mathcode`\:=\string"8000
\begingroup \catcode`\:=\active
  \gdef:{\mathrel{\mathop\ordinarycolon}}
\endgroup

\newbox\mybox

\def\arrover#1{\mathrel{
\setbox\mybox=\hbox spread 1.4em{\hfil$\scriptstyle#1$\hfil}
\vbox{\offinterlineskip\copy\mybox \hbox
to\wd\mybox{\rightarrowfill}}}}

\def\larrover#1{\mathrel{
\setbox\mybox=\hbox spread 1.4em{\hfil$\scriptstyle#1$\hfil}
\vbox{\offinterlineskip\copy\mybox \hbox
to\wd\mybox{\leftarrowfill}}}}

\def\ontoover#1{\mathrel{
\setbox\mybox=\hbox spread 1.4em{\hfil$\scriptstyle#1$\hfil}
\vbox{\offinterlineskip\copy\mybox \hbox
to\wd\mybox{\rightarrowfill\hskip-2.8mm $\rightarrow$}}}}

\def\leftontoover#1{\mathrel{
\setbox\mybox=\hbox spread 1.4em{\hfil$\scriptstyle#1$\hfil}
\vbox{\offinterlineskip\copy\mybox \hbox
to\wd\mybox{$\leftarrow$\hskip-2.8mm \leftarrowfill}}}}

\def\into{\hookrightarrow}

\newdir^{ (}{{}*!/-3pt/\dir^{(}}
\newdir_{ (}{{}*!/-3pt/\dir_{(}}